\documentclass[reqno, 12pt]{amsart}

\title{The complete picture for clique factors in randomly perturbed graphs}
\author{ Sylwia Antoniuk} 
\address{Department of Discrete Mathematics, Faculty of Mathematics and CS, Adam Mickiewicz University,
	Poznań, Poland}
\email{sylwia.antoniuk@amu.edu.pl}
\author {Nina Kam\v{c}ev}
\address{Department of Mathematics, Faculty of Science, University of Zagreb, Zagreb, Croatia}
\email{nina.kamcev@math.hr}
\author{
	Christian Reiher}
\address{Fachbereich Mathematik, Universität Hamburg, Hamburg, Germany}
\email{christian.reiher@uni-hamburg.de}
\author{Tadej Petar Tukara
}\address{Department of Mathematics, Faculty of Science, University of Zagreb, Zagreb, Croatia}
\email{tadej.tukara@math.hr}

\thanks{S.~Antoniuk was supported by Narodowe Centrum Nauki, grant 2024/53/B/ST1/00164.}
\thanks{N.~Kam\v cev and T.P.~Tukara are supported by the Croatian Science Foundation, project number HRZZ-IP-2022-10-5116 (FANAP)}

\makeatletter
\let\origsection=\section \def\section{\@ifstar{\origsection*}{\mysection}} 
\def\mysection{\@startsection{section}{1}\z@{.7\linespacing\@plus\linespacing}{.5\linespacing}{\normalfont\scshape\centering\S\hspace{1pt}}}

\usepackage{amsmath,amssymb,amsthm}
\usepackage{mathrsfs}
\usepackage{mathabx}\changenotsign
\usepackage{dsfont}
\usepackage{url}
\usepackage{colortbl}

\usepackage{xcolor}
\usepackage[backref]{hyperref}
\hypersetup{
	colorlinks,
	linkcolor={red!60!black},
	citecolor={green!60!black},
	urlcolor={blue!60!black}
}

\usepackage[open,openlevel=2,atend]{bookmark}

\usepackage{doi}

\usepackage{natbib}
\bibpunct[, ]{[}{]}{,}{n}{}{,}%

\usepackage{hyperref}
\usepackage[T1]{fontenc}
\usepackage{lmodern}
\usepackage[babel]{microtype}
\usepackage[english]{babel}

\usepackage{tikz}
\usetikzlibrary{calc,decorations.pathmorphing,decorations.pathreplacing}
\pgfdeclarelayer{background}
\pgfdeclarelayer{foreground}
\pgfdeclarelayer{front}
\pgfsetlayers{background,main,foreground,front}

\usepackage{geometry}
\numberwithin{equation}{section}
\numberwithin{figure}{section}

\usepackage{enumitem}

\theoremstyle{plain}
\newtheorem{thm}{Theorem}[section]
\newtheorem{fact}[thm]{Fact}

\newtheorem{claim}[thm]{Claim}

\newtheorem{cor}[thm]{Corollary}
\newtheorem{lemma}[thm]{Lemma}

\newtheorem{remark}[thm]{Remark}

\theoremstyle{definition}

\newtheorem{definition}[thm]{Definition}

\let\eps=\varepsilon
\let\theta=\vartheta
\let\rho=\varrho
\let\phi=\varphi

\def\NN{\mathds N}

\def\QQ{\mathds Q}

\def\cA{{\mathcal A}}

\def\cC{{\mathcal C}}

\def\cF{{\mathcal F}}

\def\cK{{\mathcal K}}

\def\cP{{\mathcal P}}

\def\cS{{\mathcal S}}

\def\cU{{\mathcal U}}

\def\fA{\mathfrak{A}}
\def \fD{\mathfrak{D}}
\def\fB{\mathfrak{B}}

\def\fS{\mathfrak{S}}
\def\fU{\mathfrak{U}}

\usepackage{enumitem}

\renewcommand\labelenumi{(\roman{enumi})}
\renewcommand\theenumi\labelenumi

\def\Ub{\mathbf{U}}

\newcommand{\Bin}{\textrm{Bin}}

\newcommand{\pr}[1]{\mathbb{P} \left[ #1 \right]}

\newcommand{\E}{\mathbb{E} }
\def \Gnp{G_{n,p}}

\let\polishlcross=\l
\def\l{\ifmmode\ell\else\polishlcross\fi}

\makeatletter
\def\moverlay{\mathpalette\mov@rlay}
\def\mov@rlay#1#2{\leavevmode\vtop{   \baselineskip\z@skip \lineskiplimit-\maxdimen
		\ialign{\hfil$\m@th#1##$\hfil\cr#2\crcr}}}
\newcommand{\charfusion}[3][\mathord]{
	#1{\ifx#1\mathop\vphantom{#2}\fi
		\mathpalette\mov@rlay{#2\cr#3}
	}
	\ifx#1\mathop\expandafter\displaylimits\fi}
\makeatother

\newcommand{\dcup}{\charfusion[\mathbin]{\cup}{\cdot}}

\newcommand{\vrhup}[1]{\scaleobj{0.6}{\scalerel*{\rightharpoonup}{#1}}}
\newcommand{\nrhup}{\mathord{\scaleobj{0.6}{\scalerel*{\rightharpoonup}{x}}}}
\newcommand{\wrhup}{\scaleobj{0.6}{\scalerel*{\rightharpoonup}{W}}}
\def\vseq#1{\ThisStyle{  \mathord{\vbox{\offinterlineskip\ialign{    \hfil##\hfil\cr
					$\SavedStyle{}_{\smash{\vrhup#1}}$\cr
					\noalign{\kern-0.7\scriptspace}
					$\SavedStyle#1$\cr}}}}}
\def\seq#1{\ThisStyle{  \mathord{\vbox{\offinterlineskip\ialign{    \hfil##\hfil\cr
					$\SavedStyle{}_{\smash{\nrhup}}$\cr
					\noalign{\kern-0.5\scriptspace}
					$\SavedStyle#1$\cr}}}}}
\def\wseq#1{\ThisStyle{  \mathord{\vbox{\offinterlineskip\ialign{    \hfil##\hfil\cr
					$\SavedStyle{}_{\smash{\wrhup#1}}$\cr
					\noalign{\kern-0.7\scriptspace}
					$\SavedStyle#1$\cr}}}}}

\def \absorber{\cA}

\let\setminus=\smallsetminus
\let\emptyset=\varnothing

\let\to=\lra

\makeatletter
\newcommand{\pushright}[1]{\ifmeasuring@#1\else\omit\hfill$\displaystyle#1$\fi\ignorespaces}
\newcommand{\pushleft}[1]{\ifmeasuring@#1\else\omit$\displaystyle#1$\hfill\fi\ignorespaces}
\makeatother

\let\N=\NN
\let\R=\RR

\let\Pc=\cP
\let\Jn=J

\begin{document}

	
		
		
    

    \begin{abstract}
        A \emph{randomly perturbed graph} $G^p = G_\alpha \cup \Gnp$ is obtained by taking a~deterministic $n$-vertex graph $G_\alpha = (V, E)$ with minimum degree $\delta(G)\geq \alpha n$ and adding the edges of the binomial random graph $\Gnp$ defined on the same vertex set $V$. For which value $p$ (depending on $\alpha$) does the graph $G^p$ contain a $K_r$-factor -- a spanning collection of vertex-disjoint copies of $K_r$ -- with high probability?
        The  order of magnitude of the \textit{minimum} such $p$ is known whenever $\alpha \neq 1- \frac{s}{r}$ for an integer $s$ (see Han, Morris, and Treglown [RSA, 2021] and Balogh, Treglown, and Wagner [CPC, 2019]).
        In earlier work, the first three authors determined this threshold probability $p_s$ up to a constant factor for all values of $\alpha = 1-\frac{s}{r}\leq \frac 12$. Here, we complete the picture by establishing $p_s$ in the remaining case $\alpha > \frac12$. 

        A key ingredient in our approach is an extremal result of independent interest: we prove a fractional stability version of a \emph{tiling} theorem due to Shokoufandeh and Zhao.
        
    \end{abstract}

	\maketitle

	
	
	\section{Introduction}


    Randomly perturbed graphs have been introduced by Bohman, Frieze and Martin~\cite{bfm03,bfm04}. Since their introduction, they have been studied extensively because they interpolate between purely random graphs and highly structured deterministic graphs, providing a~useful framework for understanding fundamental graph properties such as Hamiltonicity and the existence of other spanning structures (see, for example~\cite{adrrs21,balog16,bhkm19,bmpp20,jk20,kks17,kst06}). In this paper, we further advance the study of clique factors in randomly perturbed graphs.
    
    Formally, a \emph{randomly perturbed graph} $G^p = G_\alpha \cup \Gnp$ is obtained by taking a deterministic $n$-vertex graph $G_\alpha = (V, E)$ satisfying minimum degree condition $\delta(G)\geq \alpha n$ and adding to it the edges of the binomial random graph $\Gnp$ defined on the same vertex set $V$. Given an increasing graph property $\mathcal{P}$, for any $\alpha\in(0,1)$ one would like to establish the \emph{threshold} density $p=p(n)$ for this property, that is the lowest density\footnote{As usual, this threshold is not uniquely defined and we are only interested in determining the threshold up to a constant factor for our specific property.} $p$ which, \emph{with high probability} (w.h.p.), ensures that $G_\alpha \cup \Gnp$ satisfies $\mathcal{P}$. Here, as usual, w.h.p.~means with probability tending to 1 when $n\rightarrow\infty$.

    A \emph{$K_r$-factor} in a graph $G$ is a collection of vertex disjoint copies of the complete graph $K_r$ covering all the vertices of $G$. Throughout the paper, we will assume that the number of vertices of the \emph{host graph} $G_\alpha$, denoted by $n$, is divisible by $r$. Depending on $\alpha >0$, we are interested in  determining the threshold density for the existence of a $K_r$-factor in $G_\alpha \cup \Gnp$. Both `extreme cases' ($p=0$ and $\alpha=0$) constitute classical and difficult results. Firstly, resolving a~conjecture of Erd\H os, Hajnal and Szemer\'edi \cite{Hajnal1970} have shown that for sufficiently large $n$, any $n$-vertex graph $G$ with $\delta(G) \geq \left(1-\frac{1}{r} \right)n$ contains a $K_r$-factor. Hence, for $\alpha \geq 1- \frac 1r$, no random edges are needed. On the other hand, a celebrated result of Johansson, Kahn and Vu~\cite{jkv08} states that $p = Cn^{-2/r}(\log n)^{2/(r(r-1))}$ is the threshold for the containment of a $K_r$-factor in $\Gnp$. Hence, this is the desired threshold for the case $\alpha=0$.

    Considering the $K_r$-factor problem in the model of randomly perturbed graphs reveals surprising phenomena, and requires an interesting synergy between combinatorial and probabilistic techniques. The threshold for the appearance of a $K_r$-factor in the randomly perturbed graph $G_\alpha \cup \Gnp$ has been determined for almost all values of $\alpha$: first for $\alpha \in \left(0,\frac{1}{r}\right)$ by Balogh, Treglown, and Wagner~\cite{btw18}, and subsequently for the remaining intervals of the form $\alpha \in \left(1-\frac{s}{r}, 1-\frac{s-1}{r}\right)$ by Han, Morris, and Treglown~\cite{hmt21}.
	
	\begin{thm}
		[Clique-factors in randomly perturbed graphs~\cite{btw18}, \cite{hmt21}]
		\label{thm:hmt}
		For all integers $1 < s \leq r$ and any $\alpha \in \left(1-\frac{s}{r}, 1-\frac{s-1}{r} \right)$, there exists $C$ such that the following holds. For any $n$-vertex graph $G_\alpha$ with minimum degree at least $\alpha n$, w.h.p.\ $G_\alpha \cup \Gnp$ with $p = Cn^{-2/s}$ contains a $K_r$-factor.
	\end{thm}

    The density $p = Cn^{-2/s}$ is necessary, and understanding the extremal example $G_\alpha$ sheds light on the problem. 
    In particular, the extremal example demonstrates that for $\alpha \in \left(1-\frac{s}{r}, 1-\frac{s-1}{r} \right)$, the bottleneck for the threshold density $p$ is the existence of $\Omega(n)$ copies of $K_{s}$ in the random graph $\Gnp$. Namely, let $G_{\alpha}$ consist of an independent set $A$ of size $(1-\alpha) n > \frac {(s-1)n}{r}$, and all the remaining edges (see Figure~\ref{fig:extremal_1}). If $G_{\alpha} \cup \Gnp$ is to contain a $K_r$-factor $\cF$, then $\cF$ cannot contain only cliques $K$ with $|K \cap A| \leq s-1$; in fact, at least $\Omega(n)$ cliques from $\cF$ have to intersect $A$ in at least $s$ vertices. Thus, $\Gnp[A]$ has to contain $\Omega(n)$ copies of $K_{s}$, which in turn requires edge density $p = \Omega (n^{-2/s})$.

    \begin{figure}
		\begin{tikzpicture}[scale=0.8]
			\centering
			\filldraw[color=lightgray] (0,1.3) -- (4,1.6) -- (4,-1.6) -- (0,-1.3) -- (0,1.3);
			\draw[fill=white] (0,0) ellipse (1.1cm and 1.9cm);
			\draw[fill=lightgray] (4,0) ellipse (1.3cm and 2.1cm);
			\node at (0,-2.2) [below] {$A$};
			\node at (4,-2.2) [below] {$V\setminus A$};
		\end{tikzpicture}
		\caption{\label{fig:extremal_1} The extremal construction $G_{\alpha}$ which, crucially, contains an independent set of size $(1-\alpha)n$.}
	\end{figure}
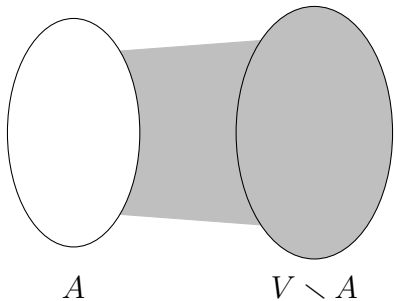

    What about the \textit{transition points}, that is $\alpha$ of the form $\alpha = 1-\frac sr$\,?
    B{\"o}ttcher, Parczyk, Sgueglia, and Skokan~\cite{bpss23} considered the only interesting transition point for $r=3$ (the triangle factor), $\alpha = \frac 13$, showing that the threshold for a $K_3$-factor in $G_{\frac 13} \cup \Gnp$ is $p = \frac{C \log n}{n}$. Instructed by this result, the example from Figure~\ref{fig:extremal_1} and the Johansson--Kahn--Vu Theorem, one might guess the following: as $\alpha$ decreases from $\alpha > 1-\frac{s}{r}$ to precisely $\alpha = 1- \frac sr$, the threshold density `gains' a~polylogarithmic factor.


    
    However, Böttcher et al.~\cite{bpss23} have already pointed out a rather surprising polynomial jump in the threshold probability when $r=4$ and $s=3$ ($\alpha=1/4$), and asked about the correct threshold. In~\cite{akr24}, the first three authors have pinpointed this threshold, and generalised it to an infinite family of transition points. The table below summarises the thresholds for $r=4$.

    \smallskip
    
    \begin{center}
           \begin{tabular}{| c | c |c | >{\columncolor{gray!15}}c | c|c| c | c |}
    \hline
    & & & & & & & \\[-3ex]
    $\alpha$ & 0 & $\left(0, \frac{1}{4}\right)$ & $\frac{1}{4}$ & $\left(\frac{1}{4}, \frac{1}{2}\right)$ & $\frac{1}{2}$ & $\left(\frac{1}{2}, \frac{3}{4}\right)$ & $\left[\frac{3}{4}, 1\right]$ \\ [0.5ex]
    \hline
    & & & & & & & \\[-3ex]
    $K_4\text{-factor}$ threshold & $n^{-\frac{2}{4}} \log^{\frac{1}{6}} n$ & $n^{-\frac{2}{4}}$ & $n^{-\frac{3}{5}}$ & $n^{-\frac{2}{3}}$ & $n^{-1} \log n$ & $n^{-1}$ & $0$ \\ [0.4ex]
    \hline
  \end{tabular}
      \end{center}

    \smallskip
    

    Specifically, in~\cite{akr24}, the transition points $\alpha = 1-\frac sr$ were considered for all $r \geq 3$ and $1<s<r$. The threshold was determined whenever $s \geq r/2$ (equivalently, $\alpha \leq 1/2$), and a~lower bound was found for all values of $s$. 

    The extremal construction establishing the lower bound from~\cite{akr24} is a partly pseudorandom host graph $G_\alpha$. This is a completely new approach, as in most results on the randomly perturbed graphs the 0-statements are proved using some highly structured graphs such as complete multipartite graphs. Rather than describing the construction here, we refer the reader to~\cite{akr24}.
    
    Before stating the results, we introduce the crucial functions. Let 
	\(\varphi(s) = \frac{2s}{(s-1)(s+2)}\)
	and 
	\[p_s = p_s(n) = \begin{cases} 
		\frac{\log n}{n} & \text{ if } s=2 \\
		n^{-\varphi(s)} & \text{ if } s\geq 3.
	\end{cases} \]

    The following theorem was proved in~\cite{akr24}. For $r=3$ and $s=2$, both statements were earlier established by B{\"o}ttcher, Parczyk, Sgueglia, and Skokan~\cite{bpss23}.
    \begin{thm}[Clique-factors at the transition points \cite{akr24}]
    \label{thm:main_in_AKR}
		Let $r$ and $s$ be integers with $1 < s <r$, and let $n$ be divisible by $r$.
        \begin{enumerate}
            \item For any $\eps>0$, there is a~constant $c>0$ and a graph $G$ with minimum degree at least $\left(1 - \frac{s}{r} \right)n$ such that taking $p = cp_s(n)$ the probability that the graph $G^p = G \cup \Gnp$ has a~$K_r$-factor is less than $\eps$.
            \item If $s \geq r/2$, there exists $C>0$ such that for $p = Cp_s(n)$ the following holds. If $G$ is an $n$-vertex graph with minimum degree at least $\left(1 - \frac{s}{r} \right)n$, then  w.h.p.\,the graph $G^p = G \cup \Gnp$ has a~$K_r$-factor.
        \end{enumerate}
	\end{thm}
	 
	In~\cite{akr24} it was conjectured that statement (ii) extends to all values of $s$. In the present paper, we confirm this conjecture, also answering the questions raised in~\cite{bpss23,hmt21}. 

    \begin{thm}\label{t:main}
		For any integers $r$ and $s$ with $1 < s < r$, there exists $C>0$ such that the following holds. Let $G$ be an $n$-vertex graph with minimum degree at least $\left(1 - \frac{s}{r} \right)n$. For $n$ divisible by $r$ and $p = Cp_s(n)$, w.h.p.\,the graph $G^p = G \cup \Gnp$ has a~$K_r$-factor.
	\end{thm}

    In the remainder of the introduction, we highlight an extremal result that is of independent interest, and other proof ideas that require treating the case $s \geq r/2$ in a separate article.

    \subsection*{A Dirac--type problem}
       
    How does the Hajnal--Szemere\'edi Theorem generalise to $H$-factors, for a given graph $H$\,? Shokoufandeh and Zhao~\cite{sz03} proved that if $\delta(G) \geq \left(1-\frac{1}{\chi_{cr}(H)} \right)n$, then $G$ contains a collection of vertex-disjoint copies of $H$ covering all but a constant number of vertices; here, $\chi_{cr}(H)$ is the \textit{critical chromatic number} of $H$. This result was previously conjectured by Koml\'os~\cite{komlos00}. The minimum degree condition is optimal for every graph $H$, and the presence of a constant number of uncovered vertices is unavoidable (see K\"uhn, Osthus~\cite{ko09}).

    We show the \textit{fractional stability version} of~\cite{sz03}. To be precise, a \emph{weighted graph} is a pair $(F, w_F)$, where $F$ is a~graph and $w_F: V(F) \to \R_{\geq 0}$ is a~weight function.
    For a~family of weighted graphs $\cF$, an \emph{$\cF$-packing} in $G$ is a~collection of graphs $(F_i)_{i \in I}$, where $F_i \in \cF$, with corresponding embeddings $(f_i)_{i \in I}$ into $G$, such that for any $x \in V(G)$, 
    $$\sum_{i \in I: Im (f_i) \ni x} w_{F_i}(f_i^{-1}(x)) \leq 1.$$
    In other words, for each $x \in V(G)$ the total weight in $(F_i)_{i\in I}$ of all vertices mapped to $x$ is at most 1. For now, we only consider the case $\cF=\{H\}$, and thus refer to $H$-packings\footnote{The terms fractional packing, fractional tiling and partial fractional factor are also often used in the literature.}.
    For a weighted graph $H$ and $\phi > 0 $, denote by $\phi\ltimes H$ the weighted graph obtained by multiplying all the vertex weights of $H$ by $\phi$.

    The following theorem states that if $G$ is `sufficiently far' from the extremal example in the Shokoufandeh--Zhao theorem, then an $H$-packing with `constant residue' can be found, even under a relaxed minimum-degree assumption.

        \begin{thm}\label{t:sz-stability}
        Let $H$ be a graph with critical chromatic number $\chi_{cr}(H)$.  There are integers $b,C>0$ such that for all $\gamma>0$, there exists a constant $\delta >0$ such that the following holds for sufficiently large $n$. If $G$ is an $n$-vertex graph with $\delta(G)\geq \left(1-\frac{1}{\chi_{cr(H)}}-\delta\right)n$ in which every vertex set of size $\frac{n}{\chi_{cr}(H)}$ spans at least $\gamma n^2$ edges, then $G$ contains a $(b^{-1} \ltimes H)$-packing $(f_i)_{i \in I}$ with 
         $$n- |I||V(H)|b^{-1}\leq C.$$\end{thm}
        An asymptotic version of this statement was shown by Hladk\'y, Hu and Piguet~\cite[Theorem 1.3]{hhp19}. However, in the result of~\cite{hhp19}, the $H$-copies can cover all but $O(\delta n)$ vertices (where, crucially, $\delta$ also appears in the minimum-degree assumption on $G$).

        Theorem~\ref{t:sz-stability} is the key to our proof, and it might have further applications in extremal combinatorics. It is proved in Section~\ref{sec:sz-stability}.
        







    \subsection*{Proof ideas} 
    Let us mention a few proof ingredients, which also justifies treating the cases $s \geq r/2$ and $s< r/2$ separately. In short, the present proof requires substantial new results and ideas, including Theorem~\ref{t:sz-stability}. On the other hand, the present proof does not imply the result of~\cite{akr24}. Finally, part of the article~\cite{akr24} was devoted to developing tools for constructing `spread embeddings' in regular partitions and applying the recently proved (fractional) Kahn--Kalai conjecture~\cite{fknp21,pp24}. We now use these tools (Lemma~\ref{l:Z-conf.-Kt-factor}) as a~black box.

    Recall that $r = ms+t$ with $t<s$, and that we have to prove the statement for every deterministic graph $G$ with $\delta(G) \geq 1 - \frac sr$. We distinguish a few cases depending on the structure of $G$. Note that the extremal example is \textit{close to} the graph in Figure~\ref{fig:extremal_1}, where $|A| = \frac{sn}{r}-o(n)$, and $G[A]$ induces a sparse random-like graph~\cite{akr24}. 

    Most of the work and innovation lies in the \textit{non-extremal case}, where $G$ does not contain an almost independent set of size roughly $\frac{sn}{r}$. In this case, we apply Szemer\'edi's Regularity Lemma~\cite{szemeredi78}, and wish to find a useful substructure in the \textit{reduced} graph. The useful structure is essentially given by Theorem~\ref{t:sz-stability}.

    Even upon finding this structure, we are not done using standard techniques. Namely, a~lot of work is devoted to covering the leftover vertices and ensuring exactly the desired part sizes in the regular partition. This is also a major challenge in~\cite{bpss23,akr24}, and whenever exact minimum degree conditions are needed. The absorption strategy for resolving this issue is different from~\cite{akr24}, and it relies heavily on the fact that, since $\delta(G) > \frac 12$, the graph is connected, and any pair of vertices has many common neighbours (see, e.g.~Definition~\ref{def:absorber}).

    In the extremal case, the host graph $G$ contains an almost independent set $A$ of size roughly $\frac{sn}{r}$, and, due to the minimum degree condition, most edges between $A$ and $V\setminus A$ are present. Therefore, in constructing a $K_r$-factor, most of the copies of $K_r$ use exactly $s$ vertices from $A$, and the respective sets should span cliques in the random graph $\Gnp$. However, a number of copies of $K_{s+1}$ in $G[A]$ may also be required, which is the bottleneck for our construction and which in turn determines the threshold density $p_s$. 

     An additional difficulty in the present paper are \textit{intermediate cases}. For instance, suppose that $G$ consists of an independent set $A$ of size roughly $\frac{sn}{r}$, but the remaining graph $G[V \setminus A]$ does not have a particular structure. The main idea for resolving this case is an inductive approach: recalling that $G^p = G \cup G_{n,p}$, we can find a $K_{r-s}$-factor in $G^p[V \setminus A]$ (which roughly satisfies the \textit{correct} minimum-degree condition and is \textit{non-extremal}), and a $K_s$-factor in $G^p[A]$. The aim and the difficulty is then to construct those factors in a way that they can be combined into a $K_r$-factor. An additional hurdle is that $|A|$ can differ from $\frac{sn}{r}$. Most of these issues are addressed in Section~\ref{sec:pf-of-main}. In particular, a~stronger conclusion is needed in the \textit{non-extremal} case (Theorem~\ref{t:non-extremal-conforming}).
    
    Finally, a careful reader may notice that the number of possible cases is roughly $\lfloor r/s \rfloor$, and the case distinction (Lemma~\ref{l:4.3chr}) is rather refined.

    \subsection*{Paper structure} 
    The article is organised as follows. Preliminary definitions and results are stated in Section~\ref{sec:prelim}. Section~\ref{sec:pf-of-main} contains the main lemma statements which roughly correspond to the outlined case distinction, as well as a deduction of the main result, Theorem~\ref{t:main}. The afore-mentioned lemmas are proved in Sections~\ref{sec:non-extremal} and~\ref{sec:partition-lemma}. Section~\ref{sec:sz-stability} contains the proof of Theorem~\ref{t:sz-stability}, and further extremal consequences of this result.	
	
	
    \section{Preliminaries}
	   \label{sec:prelim}
    Let $m, s, t$ be integers such that $r = ms + t$ with $1\leq t \leq s$. Throughout the paper, we refer to the case $s = t = r/(m+1)$ as the \textit{singular} case, and in this case some steps are treated separately. We will often use an auxiliary parameter $g$ defined as
    \[ g = \begin{cases}
        0 & \text{if} \quad t<s; \\
        1 & \text{if} \quad t=s.
    \end{cases}\]

    We use the same notation as in \cite{akr24}. In particular, for any set $X$ and an integer $\ell$, by ${X \choose \ell}$ we denote the family of all $\ell$-element subsets of $X$.
    For a collection of graphs $\cC$, by $V(\cC)$ we denote the set of vertices of the graphs in $\cC$. Let $e_G(X,Y)$ be the number of edges in $G$ whose one endpoint lies in $X$ and the other in $Y$. If $X\cap Y\neq\emptyset$, then any edge $uv$ with $u,v\in X\cap Y$ is counted twice in $e_G(X,Y)$.

    \begin{definition}
        For a~vertex set $Z$, we call a collection $\cK$ (usually a $K_r$-factor) \textit{$Z$-conforming} if for all $K \in \cK$, $K$ contains at most one element of $Z$.
    \end{definition}

	
	
	
    \subsection{Probabilistic bounds}

    The following lemma is proved in \cite{akr24} for $s\geq 3$, and for $s=2$ follows from classic results on the independence number of $\Gnp$ (found, for instance, in~\cite[Theorem 7.4]{jlr00}).

    \begin{lemma}\label{l:ks-in-lin-sized}
        For $s \geq 2$ and $\eps >0$, let  $p = n^{-2/s+\eps}$. Then, w.h.p.\,any subset $S \subset V(\Gnp)$ with $|S| \geq \frac{n}{\log^2 n}$ contains a copy of $K_s$.
    \end{lemma}

    The next lemma is a well-known consequence of Janson's Inequality.
    
    \begin{lemma}
	\label{l:akr:ks-s-everywhere}
        Let  $s$ be an integer, and $\nu >0$. Let $V$ be an $n$-vertex set, and let $\cC$ be a~family of $s$-element subsets of $V$ with $|\cC|\geq \nu n^{s}$. For $p = Cn^{-2/s}$ and $C$ sufficiently large, the probability that no member of $\cC$ induces a copy of $K_s$ in $G(V,p)$ is at most $e^{-c n^2 p}$, where $c$ is a constant depending on $s$.
    \end{lemma}

    The following lemma is a generalisation of the previous one, in which a slightly more complex subgraph is sought in $\Gnp$.
    
    \begin{lemma}
        \label{l:many_ks_and_one_ks-ks}
        Let $s,q\ge 2$ be integers, and let $\nu>0$.  
Define $F_q$ to be the graph on vertex set $[qs-1]$ in which the sets
$\{1,\dots,s\}$ and $\{s,\dots,2s-1\}$ induce copies of $K_s$, and for each
$i=3,\dots,q$, the set $\{(i-1)s,\dots,is-1\}$ induces a copy of $K_s$,
with no other edges. (Thus $F_q$ consists of $q$ copies of $K_s$, exactly two
sharing exactly one vertex, and all others pairwise vertex-disjoint.)

Let $V$ be an $n$-vertex set and let $\cC\subset V^{qs-1}$ satisfy
$|\cC|\ge \nu n^{qs-1}$. For $p=Cn^{-2/s}$ and $C$ sufficiently large, the probability that no member of $\cC$ induces a copy of $F_q$ in $G(V,p)$ is at most $e^{-c n^2 p}$, where $c=c(s,q)>0$ is a constant.
    \end{lemma}
    \begin{proof}
        [Proof sketch.] We apply induction on $q$, noting that $F_{q+1}$ is isomorphic to the disjoint union of $F_q$ and $K_s$.
        
        Suppose first that $q=2$. For this case, the lemma was essentially proved in~\cite[Lemma 9.2]{akr24}. It is an application of Janson's Inequality, and even though the setup in~\cite{akr24} is slightly more specific, the calculation extends straightforwardly to the present statement.

        Now assume that the Lemma holds for $q$, and let $\cC \subset V^{(q+1)s-1}$ be as in the statement. Using double exposure, we sample $\Gnp$ as $G_1\cup G_2$, where $G_1, G_2 \sim G_{n,p_1}$ are independent and $1-p = (1-p_1)^2$. Applying a standard pruning argument and Lemma~\ref{l:akr:ks-s-everywhere}, w.h.p.~$G_1$ contains an $s$-tuple $\textbf{v}$ 
        which induces a $K_s$-copy and extends to at least $\nu'n^{qs-1}$ members of $\cC$. Then we can apply the induction hypothesis to this subcollection of $\cC$ extending $\textbf{v}$ and the graph $G_2$, finding a copy of $F_{q+1}$ in $G_1 \cup G_2$, as desired.
    \end{proof}

    The next lemma, proved in \cite{akr24}, is the key ingredient in the main proof and turns out to be the bottleneck for the threshold probability $p_s$ when $s\geq 3$.
	
     \begin{lemma}[\cite{akr24},        Lemma~4.3]\label{l:5.5chr}
        For all $s\geq 2$, there are constants $\gamma, C>0$ such that the following holds. Let $G$ be an $n$-vertex graph with minimum degree $h\leq \gamma n$. W.h.p.\,the graph $G^p = G \cup \Gnp$, with $p = Cp_s(n)$, contains a collection of $h$ vertex-disjoint copies of $K_{s+1}$. 
    \end{lemma}

    In the current paper we will need the following $Z$-conforming strengthening of Lemma~\ref{l:5.5chr}.
    
    \begin{lemma}\label{l:5.5chr-stronger}
        For all $s\geq 2$ and $\delta\ll\beta\ll 1/s$, there are constants $\gamma, C>0$ such that the following holds. Let $G$ be an $n$-vertex graph with minimum degree $h\leq \gamma n$. Let $Z\subset V(G)$ be such that\footnote{In the application, these will be vertices with degree less than $|A_i|-cn$ into some part $A_i$, where $c$ is some small constant like $c <\beta r^{-3}$. $\beta$ is the constant from the partitioning lemma.} $|Z| \leq \delta n$, and suppose that for $v \in Z$, $\deg_G(v) \geq \beta n$. W.h.p.\,the graph $G^p = G \cup \Gnp$, with $p = Cp_s(n)$, contains a collection of $h$ vertex-disjoint copies of $K_{s+1}$ such that each of those copies contains at most one vertex of $Z$.
    \end{lemma}
    
    \begin{proof}
        Assume first that $|Z|\geq h$. Then we may greedily find a collection of vertex-disjoint copies of $K_{s+1}$, each using exactly one vertex from $Z$. Indeed, we find copies of $K_{s+1}$ one-by-one. At the point of processing a vertex $v\in Z$, its number of neighbours in $G$ distinct from vertices in all previous $K_{s+1}$-copies is at least 
        \begin{equation}
                \beta n-|Z|(s+1)\geq n\left(\beta-(s+1)\delta\right)\geq \frac{\beta n}{2}
            \label{eq:Z-min-degree}
        \end{equation}
        for sufficiently small $\delta$. Thus, by Lemma~\ref{l:akr:ks-s-everywhere}, the probability that there is no $K_s$-copy in $N_G(v)\setminus Z$, vertex-disjoint from all previous copies, is at most $e^{-cn}$ for a sufficiently small constant $c>0$. Taking the union bound, the probability that we will fail at any stage of this process is $o(1)$.
        
        Suppose now that $|Z|<h$. Note that $G-Z$ has minimum degree at least $h-|Z|\leq\gamma n$, so we may apply Lemma~\ref{l:5.5chr} to $G-Z$, which yields a collection $\cC$ of $h-|Z|$ vertex-disjoint $K_{s+1}$-copies that are disjoint from $Z$. Now, as in the above argument (noting that $|V(\cC)|\leq \gamma (s+1)n\leq \beta n/4$, so a bound slightly weaker than \eqref{eq:Z-min-degree} holds), this collection can be greedily extended to a collection of size $h$.
    \end{proof}

    \begin{lemma}\label{l:Z-conf.-Kt-factor}
        For all $s\geq t\geq 2$ and $\delta \ll \beta \ll 1/s$, there is a constant $C>0$ such that the following holds. Let $G$ be an $n$-vertex graph, with $n$ divisible by $t$, and $Z\subset V(G)$ be a~subset with $|Z| \leq \delta n$ such that for any vertex $v\in Z$, $\deg_G(v) \geq \beta n$. W.h.p.~the graph $G^p=G\cup \Gnp$, with $p=Cp_s(n)$, has a $K_t$-factor in which each clique contains at most one vertex from $Z$.
    \end{lemma}

    \begin{proof}
        Using the same approach as in the proof of Lemma~\ref{l:5.5chr-stronger}, we first greedily find a~collection $\cC$ of vertex-disjoint $K_t$-copies, each using exactly one vertex from $Z$. We then cover the remaining vertices in $V(G)\setminus V(\cC)$ by a $K_t$-factor using only the edges in $\Gnp$. We can do this since $p$ is strictly above the threshold for a $K_t$-factor (see \cite{jkv08}, Theorem~2.1).
    \end{proof}	
	

    \subsection{Regularity}\label{sec:prelim-reg}
	

    In this section, we recall some standard terminology and results regarding regularity.  
    Let $G$ be a graph, and let $V_1, V_2 \subset V(G)$ be disjoint subsets of the vertices of $G$. For non-empty sets $X_1\subseteq V_1$, $X_2 \subseteq V_2$, we define the \emph{density of $G[X_1,X_2]$} as $d_G(X_1, X_2)\coloneq \tfrac{e_G(X_1,X_2)}{|X_1||X_2|}$.
    Given $\eps, d > 0$, we say that a pair $(V_1,V_2)$ is \emph{$(\eps,d)$-regular} in $G$ if for $i = 1,2$ and for all sets $X_i \subseteq V_i$  with $|X_i|\geq \eps |V_i|$, we have $| d_G(X_1,X_2)-d| < \eps$. We say that a pair $(V_1,V_2)$ is $(\eps,d^+)$-regular if it is $(\eps,d')$-regular for some $d'\geq d$. We often omit the subscript $G$ if it is clear from the context.

    We will use the following standard lemmas, which can be found, for example, in~\cite{abcd22,ks95}. We begin by recalling the Slicing Lemma.
	
    \begin{lemma}[\cite{abcd22}, Lemma~2.7]\label{l:slicing}
        Let $0 < \eps < \beta$, $d \leq 1$ and let $(V_1,V_2)$ be an $(\eps,d)$-regular pair. Then for any pair $(U_1, U_2)$ such that for $i=1,2$, $U_i\subset V_i$ with $|U_i| \geq \beta |V_i|$, the pair $(U_1,U_2)$ is $(\eps', d')$-regular with $\eps' = \max\{\eps/\beta, 2\eps \}$ and some $d'>0$ satisfying $|d' -d|\leq \eps$.
    \end{lemma}

    The following lemma is similar. Although the proof is standard, we have not found a~self-contained statement in the literature, so we prove it here.

    \begin{lemma}\label{l:slicing-adding}
        Let $0<\xi<\eps \ll 1$, $d\leq 1$ and let $(U_1, U_2)$ be an $(\eps,d)$-regular pair. Then for any pair $(V_1,V_2)$ such that for $i=1,2$, $U_i\subset V_i$ with $|V_i\setminus U_i|\leq \xi |U_i|$, the pair $(V_1,V_2)$ is $(\eps', d')$-regular with $\eps'= \max\{\xi/\eps, 6\eps\}$ and some $d'>0$ satisfying $|d'-d|<3\eps $.
    \end{lemma}

    \begin{proof}
        Let $A_i\subseteq V_i$ with $|A_i|\geq \eps' |V_i|$. It is easy to see that $|A_i\cap U_i|\geq (\eps'-\xi)|U_i|\geq \eps |U_i|$ and $|A_i\cap U_i| \geq |A_i|-|V_i\setminus U_i|\geq |A_i|-\xi|U_i|\geq|A_i|-\xi|V_i|\geq |A_i|(1-\xi/\eps')\geq |A_i|(1-\eps).$ 
        We have 
        \begin{align*}
            e_G(A_1,A_2)&\geq e_G(A_1\cap U_1, A_2 \cap U_2) \geq d(1-\eps)|A_1\cap U_1||A_2\cap U_2|\\
            &\geq d (1-\eps)^3|A_1||A_2| \geq (d-3\eps)|A_1||A_2|
        \end{align*}
        and, since $|A_i\setminus U_i|\leq \xi|V_i|\leq \frac{\xi}{\eps'}|A_i|$,
        \begin{align*}
            e_G(A_1,A_2) &\leq e_G(A_1\cap U_1, A_2 \cap U_2)+\frac{2\xi}{\eps'}|A_1||A_2| \\
            &\leq \left(d+\eps+\frac{2\xi}{\eps'}\right)|A_1||A_2| \leq (d+3\eps)|A_1||A_2|.
        \end{align*}
        
        As for the density bound, we have 
        \[d'=\frac{e(V_1,V_2)}{|V_1||V_2|} \geq \frac{e(U_1,U_2)}{(1+\xi)^2|U_1||U_2|} \geq (1-3\xi)d \geq d - 3\xi\]
        and 
        \[d'\leq \frac{e(V_1,V_2)}{|V_1||V_2|}\leq  \frac{e(U_1,U_2)+2\xi|U_1||U_2|+\xi^2|U_1||U_2|}{|U_1||U_2|}\leq d+3\xi.\]
    \end{proof}
	
    To state the version of the Regularity Lemma which we will apply, we recall the notion of regular partitions and the reduced graph.
    We say that the partition $V_0\dcup V_1 \dcup \ldots \dcup V_\ell$ of $V(G)$ is an \emph{$\eps$-regular partition} if $|V_0| \leq \eps |V(G)|$, $|V_1| = \ldots = |V_{\ell}|$, and for all but at most $\eps \ell^2$ pairs $(i,j) \in [\ell] \times [\ell]$, $i\neq j$, the pair $(V_i, V_j)$ is $\eps$-regular. The \emph{$(\eps, d)$-reduced graph} $R$ with respect to the above partition is a graph on vertex set $[\ell]$ such that the edges $ij$ correspond precisely to the $(\eps, d')$-regular pairs $(V_i,V_j)$ with density $d' \geq d$.  
	
    \begin{lemma}[\cite{abcd22}, Lemma~2.6]\label{l:regularity}
        For all $0 < \eps \leq 1$ and $\ell_0 \in \N$, there exists $L, n_0 \in \N$ such that for every $0<d<\alpha<1$, every graph $G$ on $n > n_0$ vertices with minimum degree $\delta(G)\geq \alpha n$ has an $\eps$-regular partition $V_0 \dcup V_1 \dcup \ldots \dcup V_\ell$ with $(\eps,d)$-reduced graph $R$ on $\ell$ vertices such that $\ell_0 \leq \ell \leq L$ and $\delta(R)\geq (\alpha - d - 2\eps)\ell$.
    \end{lemma}

    Next, to find spanning subgraphs, it is convenient to strengthen the notion of regular pairs to super-regular pairs.   
    A pair of sets $(V_1, V_2)$ in a graph $G$ is \emph{$(\eps, d, \vartheta)$-super-regular} if it is $(\eps, d)$-regular, and for $i \neq j$, every vertex $v \in V_i$ has at least $\vartheta|V_j|$ neighbours in $V_j$. As usual, we refer to an $(\eps, d, d-\eps)$-super-regular pair as \emph{$(\eps, d)$-super-regular}. We will say that a $k$-tuple $(V_1,\dots, V_k)$ is $(\eps,d)$-regular (resp.~$(\eps,d,\vartheta)$-super-regular) if each pair in this tuple is $(\eps,d)$-regular (resp.~$(\eps,d,\vartheta)$-super-regular).

    The following lemma allows us to pass from regular $k$-tuples to super-regular $k$-tuples with appropriate parameters by discarding some vertices.

    \begin{lemma}[\cite{abcd22}, Lemma~2.9]\label{l:super-reg-k-tuple}
        Let $k\in\mathbb{N}$ and $0<\eps<d<1$ with $\eps\leq\frac{1}{2k}$. If $(V_1,\dots, V_k)$ is an $(\eps,d^+)$-regular $k$-tuple of vertex-disjoint sets of size $n$, then there are subsets $\tilde V_i \subseteq V_i$, $i\in[k]$ with $|\tilde V_i|=\lceil (1-k\eps)n\rceil$ for all $i\in[k]$ so that the $k$-tuple $(\tilde V_1,\dots,\tilde V_k)$ is $(2\eps, (d-\eps)^+, d-k\eps)$-super-regular.
    \end{lemma}
	
    The next lemma allows us to discard some edges in a super-regular pair and maintain super-regularity with appropriate parameters.
	
    \begin{lemma}[\cite{abcd22}, Lemma~2.12]\label{l:super-reg-min}
        Let $0<\eps<1$ and let $G$ be a bipartite graph with parts $V_1, V_2$ of size $n$, for sufficiently large $n$. Let $\vartheta \in [4 \eps, 1]$. If $(V_1, V_2)$ is $(\eps^2,d', \vartheta-\eps^2)$-super-regular, then there is a spanning subgraph $G' \subset G$ such that $(V_1,V_2)$ is~$(4\eps,\vartheta, \vartheta - 4\eps)$-super-regular in~$G'$.
	\end{lemma}
	
    \subsection{Other auxiliary lemmas}
	
    \begin{lemma}[\cite{akr24}, Lemma~7.3]
	\label{l:stronger-regular-pairs-factors}
        For integers $m, s, r$  and constants $ \vartheta \gg \eps$, there exists a constant $C$ such that the following holds. Let $G$ be an $nr$-vertex graph and $(V_{i,j})$ be a collection of $n$-vertex pairwise disjoint subsets of $G$, with $i \in [m]$ and $j \leq s^*(i) \leq s$. Suppose that for $i \neq i'$, each $(V_{i,j}, V_{i',j'})$ is an $(\eps, d,  \vartheta)$-super-regular pair in $G$ with some $d \geq \vartheta$. For $r = s^*(1)+\dots+s^*(m)$ and $p = C{n^{-2/s}}\left(\log n \right)^{2/(s(s-1))}$, w.h.p.~the graph $G \cup G_{nr, p}$ contains a $K_r$-factor in which every copy of $K_r$ uses exactly one vertex of each $V_{i,j}$.
    \end{lemma}
	
    We will need the following $Z$-conforming strengthening of the above lemma.
	
    \begin{lemma}\label{l:st-reg-pairs-factors}
        Let $r=ms+t$ with $t \in [s]$. For all $\vartheta \gg \eps $, there exists a constant $C$ such that the following holds. Let $G$ be an $(m+1)$-partite $n$-vertex graph with vertex classes $V_1 \dcup \dots \dcup V_{m+1} = V$ satisfying 
		$$|V_i| = \frac{s|V|}{r} \ \text{ for } i \in [m] \text{ \quad and \quad } |V_{m+1}| = \frac{t|V|}{r}.$$ Suppose that each pair $(V_i, V_j)$, $i \neq j$, is $( \eps, d, \vartheta)$-super-regular with $d \geq \vartheta$. Let $Z\subset V$ be given, with $|Z\cap V_i| \leq \eps |V_i|$ for $i \in [m+1]$. For $p = C{n^{-2/s}}\left(\log n \right)^{2/(s(s-1))}$, if $n=|V|$ is divisible by $r$, then w.h.p.\,the graph $G \cup \Gnp$ contains a $Z$-conforming $K_r$-factor.
    \end{lemma}
    
    \begin{proof}
        We partition the sets $V_i$ as follows. For $i\in[m]$, partition the set $V_i$ into sets $V_{i,1},\dots,V_{i,s}$ by first putting all vertices in $Z\cap V_i$ in $V_{i,1}$, and then splitting the remaining vertices of $V_i\setminus Z$ between sets $V_{i,1},\dots,V_{i,s}$, so that each of them has the same size $\frac{|V|}{r}$, uniformly at random. We partition the set $V_{m+1}$ analogously, but into $t$ sets. By Lemma~\ref{l:slicing}, all pairs $(V_{i,j}, V_{i',j'})$, $i\neq i
        '$, are $(s\eps,(d-\eps)^+)$-regular. Moreover, since we split the vertices randomly, w.h.p.~the degree of any vertex in $V_{i,j}$ into $V_{i',j'}$, $i\neq i'$, is at least $\frac{\vartheta}{2}|V_{i',j'}|$. Indeed, consider a vertex $x\in V_{i,j}$. Since the pair $(V_i, V_{i'})$ is $(\eps,d,\vartheta)$-super-regular, $x$ has at least $\vartheta|V_{i'}|$ neighbours in $V_{i'}$. Let $X\subset V_{i'}$ consist of exactly $\vartheta|V_{i'}|$ arbitrarily chosen neighbours of $x$. Then $|X\cap V_{i',j'}|$ is a hypergeometric random variable with expectation 
        \[\mu \geq \frac{(1-\eps)|V_{i',j'}|\vartheta|V_{i'}|}{|V_{i'}|} = (1-\eps)\vartheta|V_{i'j'}| = \frac{(1-\eps)\vartheta}{r}n. \]
        Thus, using Chernoff-type bounds for the hypergeometric distribution (see, e.g.~\cite{skala2013hypergeometric}), the probability that $|X\cap V_{i',j'}|<\frac{\vartheta}2|V_{i'j'}|$ is $e^{-\Omega(n)}$, so the assertion follows by a union bound over all vertices. 
        In particular, the pairs $(V_{i,j}, V_{i',j'})$, $i\neq i
        '$, are $\left(s\eps,(d-\eps)^+,\frac{\vartheta}{2}\right)$-super-regular.

        We can now use Lemma~\ref{l:stronger-regular-pairs-factors} to get the desired $Z$-conforming $K_r$-factor.
    \end{proof}


    \section{Proof of the main theorem}
        \label{sec:pf-of-main}

    In this section, we present a proof of Theorem~\ref{t:main}. Before we proceed, we state the key auxiliary results, the first two of which is the following treatment of the fully non-extremal case. 
    
    \begin{thm}
        \label{t:non-extremal-conforming}
        For any integers $r$ and $s$ with $2 \leq s < r$, let $r^{-1} \gg \gamma \gg \delta\gg \zeta \gg C^{-1} $. Let $G$ be an $n$-vertex graph with minimum degree at least $\left(1 - \frac{s}{r} - \delta\right)n$. Suppose that every subset of $V(G)$ of order $\frac{sn}{r}$ contains at least $\gamma n^2 $ edges. Let $Z\subset V(G)$ with $|Z|\leq \zeta n$ be given.  For $n$ divisible by $r$ and $p = Cp_s(n)$, w.h.p.\,the graph $G^p = G \cup \Gnp$ has a~$Z$-conforming~$K_r$-factor.
    \end{thm}

    \begin{lemma}
        \label{l:annoying-non-extremal}
        For any integers $r$ and $s$ with $2 \leq s < r$, let $r^{-1} \gg \gamma\gg \delta \gg \eta \gg C^{-1} $. Let $G$ be an $n$-vertex graph with minimum degree at least $\left(1 - \frac{s}{r} - \delta\right)n$. Suppose that every subset of $V(G)$ of order $\frac{sn}{r}$ contains at least $\gamma n^2 $ edges. For $p = Cp_s(n)$, w.h.p.\,the graph $G^p = G \cup \Gnp$ contains a collection $\cC$ of $\eta n$ vertex-disjoint copies of $K_{r+1}$.
    \end{lemma}
    
    The proofs of Theorem~\ref{t:non-extremal-conforming} and Lemma~\ref{l:annoying-non-extremal} can be found in section~\ref{sec:non-extremal}. We will also need the following lemma which extracts the structure of our initial deterministic graph $G$. To keep this section focused, we defer its proof to section \ref{sec:partition-lemma}.

        \begin{lemma}
            \label{l:4.3chr} Let $r = ms+t$ be as before. There is a constant $\beta_0>0$ such that for all $\beta \leq \beta_0$ there is a $\gamma>0$ for which the following holds. If $\gamma_1 \ll \gamma_2 \ll \ldots \ll \gamma_m \ll \gamma$ and $G$ is an $n$-vertex graph with minimum degree at least $\left( 1-\frac sr \right)n$, then for some $h \in \{0\}\cup[m]$ there exists a partition $V(G) = A_1 \dcup A_2 \dcup \dots \dcup A_{h+1}$ satisfying the following conditions.
            \begin{enumerate}
                \item\label{partition(i)} For $i \in [h]$, $|A_i| = \left(\frac{s}{r} \pm \gamma_h\right)n$.
                \item\label{partition(ii)} Every vertex $x \in V(G) \setminus A_{h+1}$ has at most $4 \beta n$ non-neighbours in $A_{h+1}$.
                \item\label{partition(iii)} Every vertex $x \in A_{h+1}$ has at least $\beta n$ neighbours in $A_i$, for any $i \in [h]$.
                \item\label{partition(iv)} If $h<m$ and $X \subset A_{h+1}$ with $|X| =sn/r$, then $e(X) > \gamma_{h+1}^2 n^2$. 
                \item\label{partition(v)} For every distinct $i, j \in [h]$, every vertex $x$ from $A_i$ has at least $\frac 13 |A_j|$ neighbours in $A_j$.
                \item\label{partition(vi)} For every distinct $i, j \in [h+1]$, there are at most $\gamma_h|A_i| |A_j|$ missing edges between $A_i$ and $A_j$.
            \end{enumerate}
        \end{lemma}

    Now we can prove the main result, that is Theorem~\ref{t:main}.

   \begin{proof}[Proof of Theorem~\ref{t:main}]
    
    Let $\beta$ be the constant given by Lemma~\ref{l:4.3chr}.
    Apply Lemma~\ref{l:4.3chr} with constants 
    \[0\ll \gamma_1 \ll  \gamma_2  \ll \ldots \ll \gamma_{m+1} \ll \beta \ll 1,\] to get a partition 
    \[V(G) = A_1 \dcup \dots \dcup A_{h+1},\]
    for some $h \in \{0\} \cup [m]$. To be more specific, $\gamma_{m}$ should be small enough for Lemma~\ref{l:st-reg-pairs-factors} to hold for $\left((2r\gamma_m)^{1/3}, \frac12, \frac{\beta}3\right)$-super-regular pairs. 
    Moreover, for any $h\in[m]$, the constants $\gamma_h$ and $\gamma_{h+1}$ should satisfy the following conditions
      \begin{equation}
        400r^4\gamma_h\ll \beta^2 \quad \text{ and } \quad
          20r^3 \gamma_h \ll \beta\gamma_{h+1}^2.
          \label{eq:gamma-h}
      \end{equation}
    Finally, $\gamma_1$ should also be sufficiently small for Theorem~\ref{t:non-extremal-conforming} to hold and the constant $C$ should be sufficiently large for all the invoked lemmas to hold. 

    In case $h=0$ (the fully non-extremal case), we are done using Theorem~\ref{t:non-extremal-conforming} with $\gamma_{\ref{t:non-extremal-conforming}}=\gamma_1$, $\zeta_{\ref{t:non-extremal-conforming}}=0$, and $Z_{\ref{t:non-extremal-conforming}}=\emptyset$.
    
    Hence, assume $h>0$. Instead of sampling $\Gnp$ with $p = Cp_s(n)$ at once, we will expose four independent graphs $G_i \sim G(n, p/5)$ for $i \in [4]$. The union $G_1 \cup G_2 \cup G_3\cup G_4$ can be viewed as a subgraph of $\Gnp$.
    Furthermore, we assume that
    \begin{itemize}
        \item[($\ast$)] any subset of $G_1$ of size at least $\frac{n}{\sqrt {\log n}}$ contains a copy of $K_s$,
    \end{itemize}
     which occurs w.h.p.~by Lemma~\ref{l:ks-in-lin-sized}.

    In order to control the vertices that have too many non-neighbours in other parts, for $i\in[h+1]$, we define the following sets 
    \[Z_i=\left\{v\in A_i : |A_j\setminus N_G(v)| > \beta |A_j| \text{ for some } j \in [h]\setminus\{i\}\right\}.\] 
    Note that by property \ref{partition(vi)} in Lemma~\ref{l:4.3chr}, we have 
     \begin{equation}
        \label{eq:zi-bound}
         |Z_i|\leq m \cdot \frac {\gamma_h}{\beta}|A_i|. 
     \end{equation}
    Let $Z=\bigcup_{i\in [h+1]}Z_i$. The following standard claim will be used throughout the proof.
        
    \begin{claim}
        \label{c:extend-by-Ks}
        Suppose that $K\subset V(G)$ is a set of vertices containing at most one vertex from $Z$, and such that $|K|\leq r$. Suppose that $K\cap A_j=\emptyset$ for some $j\in [h+1]$. Let $A_j'\subseteq A_j$ be such that $|A_j'|\geq \frac 56 |A_{j}|$. Then there is a $K_s$-copy in $G_1[(N_G(K)\cap A_j')\setminus Z_j]$. 
    \end{claim} 
         
    \begin{proof}
        First, suppose that $j\in [h]$. Since $K$ is disjoint from $A_j$, each vertex in $K\setminus Z$ has at most $\beta|A_j|$ non-neighbours in $|A_j|$. If there is a vertex in $K\cap Z$, then it has at most $\frac 23 |A_j|$ non-neighbours in $|A_j|$. Therefore $|(N_G(K)\cap A_j')\setminus Z_j|\geq |A_j|\left(\frac 56-\frac 23-r\beta-\frac{m\cdot \gamma_m}{\beta}\right)\geq \frac 17 |A_j|$. 
            
        If $j=h+1$, using property \ref{partition(ii)} from Lemma~\ref{l:4.3chr}, we similarly get that $|(N_G(K)\cap A_{h+1}')\setminus Z_{h+1}|\geq \frac 12 |A_{h+1}|$. The claim now follows from ($\ast$).
    \end{proof}

    The rest of the proof is divided into three steps. We first remove a small number of copies of $K_r$ in order to balance the sizes of the partition sets $A_i$, $i\in[h+1]$. Let $A_i' \subset A_i$ be the remaining parts. In the second step, we factorize the part induced by $A_{h+1}'$ with appropriate cliques of size $r-sh$. Finally, in the third step we extend this $K_{r-sh}$-factor to a~$K_r$-factor covering the remaining parts and such that each clique in this factor uses exactly $s$ vertices from any set $A_i'$, $i\in[h]$.

    \smallskip

    {\bf Step 1.~Balancing the sizes of the partition sets}

    \smallskip
        
    Let $g_1, g_2, \ldots, g_{h+1}$ be integers such that 
    \[|A_i|=\frac{sn}{r}+g_i, \text{ for } i\in[h] \quad \text{ and } \quad |A_{h+1}|=\frac{(m-h)s+t}{r}\cdot n+g_{h+1}.\] 
    Notice that $\sum_{i \in [h+1]}g_i = 0$ and, by property \ref{partition(i)} in Lemma~\ref{l:4.3chr}, 
    \begin{equation}\label{eq:g_i-sizes}
        |g_i| \leq \gamma_h n, \text{ for } i\in[h] \quad \text{ and } \quad |g_{h+1}|\leq h\gamma_hn.
    \end{equation}
    Let $\sigma_i=g_i/|g_i|$ be the sign of $g_i$, and set \[g=\sum_{i=1}^{h+1}|g_i|/2 \quad \text{ and } \quad I_1 = \{i\in[m+1]: g_i>0\}.\]
    Let $I_2 = [m+1]\setminus I_1$.

    \smallskip
    Consider first the case $h=m$, which is slightly easier to handle. 
    
    \begin{claim}
        \label{c:size-adj}
        If $h=m$, then w.h.p.~there is a collection $\cC$ of $g$ vertex-disjoint $Z$-conforming copies of $K_r$ in $G\cup G_1$ such that 
        \begin{itemize}
            \item for each $i\in [m]$, there are exactly $|g_i|$ cliques $K\in \cC$ with $|K\cap A_i|=s+\sigma_i$, while all other cliques $\tilde K\in \cC$ satisfy $|\tilde K\cap A_i|= s$,
            \item there are exactly $|g_{m+1}|$ cliques $K\in \cC$ with $|K\cap A_{m+1}|=t+\sigma_{m+1}$, while all other cliques $\tilde K\in\cC$ satisfy $|\tilde K\cap A_{m+1}|=t$.
        \end{itemize}
    \end{claim}
    
    \begin{proof}[Proof of Claim~\ref{c:size-adj}] 
        For any $i\in I_1\setminus\{m+1\}$ we have $\delta (G[A_i])\geq g_i$. Moreover, for any $v\in Z_i$, since there is an index $j\in [m]\setminus\{i\}$ such that $|A_j\cap N_G(v)|< (1-\beta) |A_j|$, then by the minimum-degree assumption on $G$, we have $|A_i\cap N_G(v)|\geq \beta |A_i|$. Hence, by Lemma~\ref{l:5.5chr-stronger} with $\delta_{\ref{l:5.5chr-stronger}}=m\gamma_m/\beta$, $\beta_{\ref{l:5.5chr-stronger}}=\beta$, and $\gamma_{\ref{l:5.5chr-stronger}}=\gamma_m$, w.h.p.~there is a collection $\cC_i$ of $g_i$ vertex-disjoint $Z_i$-conforming $K_{s+1}$-copies in $(G \cup G_1)[A_i]$. 
        
        If $m+1\in I_1$, we have two cases.
        
        \textbf{Case 1.} If $t<s$, let $\cC_{m+1}$ be a collection of $g_{m+1}$ vertex-disjoint $Z_i$-conforming $K_{t+1}$-copies in $(G\cup G_1)[A_{m+1}]$, which can be found greedily using ($\ast$), since $t+1 \leq s$. Here we can actually make the copies in $\cC_{m+1}$ to be vertex-disjoint from $Z_i$.

        \textbf{Case 2.} If $t=s$, by \eqref{eq:g_i-sizes} we have $|A_{m+1}| = \frac{sn}{r}+g_{m+1} \leq \frac {sn}r + m\gamma_mn$ and $\delta(G[A_{m+1}])\geq g_{m+1}$, so we can proceed to find a collection $\cC_{m+1}$ of $g_{m+1}$ vertex-disjoint $Z_{m+1}$-conforming $K_{s+1}$-copies using Lemma~\ref{l:5.5chr-stronger}, as in the case $i\in [m]$, but with $\gamma_{\ref{l:5.5chr-stronger}}=m\gamma_m$.

        Using Claim~\ref{c:extend-by-Ks}, we can now extend the cliques in $\cC = \bigcup_{i\in I_1}\cC_i$ to vertex-disjoint $K_{r+1}$-copies. To this end, we perform the following process. For each $K\in \cC_i$, go through the indices $j\in [m]\setminus \{i\}$ in order, and for each $j$, find a $K_s$-copy in $(A_j \cap N_G(K))\setminus Z_j$ (also avoiding the previously used vertices) and add it to $K$. Similarly, if $i\neq m+1$, find a $K_t$-copy in $(A_{m+1}\cap N_G(K))\setminus Z_{m+1}$ and use it to extend $K$.  This will turn $\cC$ into a collection of vertex-disjoint $Z$-conforming copies of $K_{r+1}$. Since $r|\cC|< \frac 16|A_j|$ for each $j\in [m+1]$,  Claim~\ref{c:extend-by-Ks} is applicable in each step. 

        
         Finally, since $\sum_{j\in I_2}|g_j|=g=|\cC|$, to every $K\in \cC$ we can assign an index $j_K\in I_2$ so that each index $j\in I_2$ is assigned to exactly $g_j$ cliques in $\cC$. Note that if $K\in\cC_i$, then $j_K\neq i$. Now, for each $K\in \cC$, remove from $K$ a single vertex belonging to $A_{j_K}$, to reduce a~$K_{r+1}$-copy to a $K_r$-copy. Then the resulting collection $\cC$ satisfies the conditions stated in Claim~\ref{c:size-adj}.
    \end{proof}

    Suppose now that we are in the intermediate case, i.e.~$0<h<m$. Let 
    \begin{equation}\label{eq:r'}
        r'=r-hs = (m-h)s+t \quad \text{ and } \quad \tilde{n} = \frac{r'}{r}n,
    \end{equation}
    so that $|A_{h+1}| = \tilde{n}+g_{h+1}$.
    The following auxiliary claim ensures that the graph $G[A_{h+1}]$ is fully non-extremal for appropriate parameters even after removing $O(\gamma_{h+1}^2n)$ vertices. 

    \begin{claim} \label{c:min-degree-in-A_h+1} Let $A \subset A_{h+1}$ with $||A|-\tilde{n}| \leq \frac 19 \gamma_{h+1}^2 n$. Then 
        \begin{equation*} 
            \delta(G[A])\geq\left(1-\frac{s}{r'}\right)|A| - ||A|-\tilde{n}|.
        \end{equation*}
        Secondly, if $S \subset A$ with $|S| \geq \frac{s}{r'}|A|$, then $e_G(S) \geq \frac 12 \gamma_{h+1}^2 n^2$.
    \end{claim}
    
    \begin{proof}[Proof of Claim~\ref{c:min-degree-in-A_h+1}]
        For any $x\in A$, we have
        \begin{align*}
            \deg_{G[A]}(x) & \geq \left(1-\frac{s}{r}\right)n-(n-|A|) = |A| - \frac{s}{r}n = |A| - \frac{s}{r'} \tilde{n} \\
            & \geq \left(1 - \frac{s}{r'}\right)|A| - ||A|-\tilde{n}|,
        \end{align*}
        so the minimum degree condition for $G[A]$ holds as required.

        The second part of the claim follows directly from property \ref{partition(iv)} of Lemma~\ref{l:4.3chr} if $|A| \geq \tilde{n}$, since in this case $|S|\geq\frac{s}{r'}|A| \geq \frac{s}{r}n$. 
        So assume that $|A|<\tilde{n}$.
        Let $S \subset A$ with $|S| \geq \frac{s}{r'}|A|$. Let $X$ be any subset of $A$ containing $S$ and such that $|X|= \frac{s n}{r} = \frac{s\tilde{n}}{r'}$. Note that
        $$|X \setminus S| \leq \frac{s}{r'}\left(\tilde{n}-|A|\right)\leq \frac 19 \gamma_{h+1}^2 n.$$
        Again, by property \ref{partition(iv)} of Lemma~\ref{l:4.3chr}, $X$ induces at least $\gamma_{h+1}^2 n^2$ edges in $G$. At most $n|X \setminus S| $ of those edges have an endpoint outside of $S$, so we conclude that
        $$e_G(S) \geq \frac 12 \gamma_{h+1}^2 n^2.$$
    \end{proof}
    
    Our aim now is to find a collection of $K_r$-copies analogous to that in Claim~\ref{c:size-adj}. To do this, we will expose the graphs $G_1$ and $G_2$.

    \begin{claim}
        \label{c:size-adj_h}
        If $0<h<m$, then w.h.p.~there is a collection $\cC$ of $g$ vertex-disjoint $Z$-conforming copies of $K_r$ in $G\cup G_1\cup G_2$ such that 
        \begin{itemize}
            \item for each $i\in [h]$, there are exactly $|g_i|$ cliques $K\in \cC$ with $|K\cap A_i|=s+\sigma_i$, while all other cliques $\tilde K\in \cC$ satisfy $|\tilde K\cap A_i|= s$,
            \item there are exactly $|g_{h+1}|$ cliques $K\in \cC$ with $|K\cap A_{h+1}|=r'+\sigma_{m+1}$, while all other cliques $\tilde K\in\cC$ satisfy $|\tilde K\cap A_{h+1}|=r'$.
        \end{itemize}
    \end{claim}

    \begin{proof}[Proof of Claim~\ref{c:size-adj_h}]
    Let $A = A_{h+1}\setminus Z_{h+1}$ and consider first the graph $G[A]$. Recall that by \eqref{eq:zi-bound} we have
    \[|Z_{h+1}|\leq m\cdot\frac{\gamma_h}{\beta}|A_{h+1}|.\] 
    Therefore, by Claim~\ref{c:min-degree-in-A_h+1}, we have
    $$\delta(G[A]) \geq \left(1-\frac{s}{r'} -\frac{2m \gamma_h}{\beta}\right)|A|,$$
    and any set $S \subset A$ with $|S|\geq \frac{s|A|}{r'}$ induces in $G[A]$ at least $\frac 12 \gamma_{h+1}^2 n^2$ edges. In other words, one can think of $G[A]$ as a fully non-extremal graph.
    
    We now expose the random graph $G_2$. By the properties above and using \eqref{eq:gamma-h}, we may apply to $G[A]$ Theorem~\ref{t:non-extremal-conforming} with $r_{\ref{t:non-extremal-conforming}}=r'$, $\delta_{\ref{t:non-extremal-conforming}}=2m\gamma_h/\beta$, $\gamma_{\ref{t:non-extremal-conforming}}=\frac12\gamma_{h+1}^2$, and $Z_{\ref{t:non-extremal-conforming}}=\emptyset$, after perhaps removing at most $r'-1$ vertices for divisibility issues. W.h.p., we then get in $(G\cup G_2)[A]$:
    \begin{itemize}
        \item a collection $\mathcal{F}$ of vertex-disjoint copies of $K_{r'}$ covering all but at most $r'-1$ vertices of $A$.
    \end{itemize}
    Similarly, in case $g_{h+1}>0$, applying to $G[A]$ Lemma~\ref{l:annoying-non-extremal} with $r_{\ref{l:annoying-non-extremal}}=r'$, $\gamma_{\ref{l:annoying-non-extremal}}=\frac12\gamma_{h+1}^2$, $\delta_{\ref{l:annoying-non-extremal}}=2m\gamma_h/\beta$ and $\eta_{\ref{l:annoying-non-extremal}}=h\gamma_h$, so that by \eqref{eq:g_i-sizes} we have $g_{h+1}\leq \eta_{\ref{l:annoying-non-extremal}}n$, w.h.p.~we get in $(G\cup G_2)[A]$:
    \begin{itemize}
        \item a collection $\cC_{h+1}$ of $g_{h+1}$ vertex-disjoint copies of $K_{r'+1}$.
    \end{itemize}
    
    The first collection will be used to extend some cliques $K$ spanned by vertices in $A_1\dcup \dots \dcup A_h$ to copies of $K_{r+1}$. The latter one will only be used in case $g_{h+1}>0$ to deal with surplus vertices in $A_{h+1}$. 

    Now we expose the graph $G_1$. Suppose first that $g_{h+1} > 0$. Using the same strategy as in the proof of Claim~\ref{c:size-adj}, we can extend cliques in $\cC_{h+1}$ to copies of $K_{r+1}$ in $G\cup G_1\cup G_2$, each having exactly $s$ vertices in $A_i$, for $i\in[h]$. Let $\cC_{h+1}^+$ denote this collection of $K_{r+1}$-copies.
    
    Next, using again the same strategy as in the proof of Claim~\ref{c:size-adj}, we can find a collection $\cC = \bigcup_{i\in I_1 \setminus \{h+1\}}\cC_i$ of vertex-disjoint $Z$-conforming $K_{hs + 1}$-copies in $(G\cup G_1)[A_1\dcup \dots \dcup A_h]$, such that any clique $K\in \cC_i$ contains $s+1$ vertices from $A_i$ and $s$ vertices from the other sets $A_j$, $j\in[h]\setminus\{i\}$.
    We will now use the partial $K_{r'}$-factor $\cF$ to turn those $K_{hs+1}$-copies into copies of $K_{r+1}$, one by one. 
    In particular, for each clique $K\in \cC$, we can find a clique $K'\in \cF$ such that the vertices in $V(K)\cup V(K')$ span a copy of $K_{r+1}$ in $G \cup G_1 \cup G_2$. Indeed, by removing $K'$ in each step, by \eqref{eq:g_i-sizes} we will have removed at most $|\cC|=g\leq h\gamma_h n$ cliques from $\cF$. On the other hand, by property \ref{partition(ii)} of Lemma~\ref{l:4.3chr}, the number of non-neighbours of $K$ in $A_{h+1}$ satisfies $|A_{h+1}\setminus N_G(K)|\leq 4\beta (hs+1)n$, and this is an upper bound for the number of cliques in $\cF$ that cannot be used as an extension for $K$. Therefore, after forbidding those cliques, as well as those cliques we have already used in the process, we still have at least 
    \[ |\cF| -(4\beta (hs+1)+h\gamma_h)n \geq \lfloor n'/r'\rfloor-(\gamma_{h+1}+4\beta (hs+1)+h\gamma_h)n\geq 1\] 
    other cliques in $\cF$, so we can find an appropriate extension $K'\in \cF$ for $K$. Let $\cC'$ be the new collection of those extended copies and set $\cC^+ =\cC'\cup \cC_{h+1}^+$.

    Finally, using again the same strategy as in the proof of Claim~\ref{c:size-adj}, we can turn the collection $\cC^+$ into the desired collection of copies of $K_r$ by carefully removing one vertex from each clique in $\cC^+$.
    \end{proof}

    Let $\cC$ be the collection of copies of $K_r$ given by Claim~\ref{c:size-adj} in case $h=m$ or by Claim~\ref{c:size-adj_h} in case $0<h<m$. We remove from $A_1\dcup\dots\dcup A_{h+1}$ the vertices belonging to cliques in $\cC$. Let $A_i' \subset A_i$, $i\in[h+1]$, be the vertex sets obtained after this removal. Let $n' = n-gr$, and note that due to the properties of the collection $\cC$ we have
    \begin{equation}\label{eq:A_{h+1}'}
        |A_i'| = \frac{s n'}{r} \text{ for } i \in [h] \quad \text{ and } \quad |A_{h+1}'| = \frac{(r-hs)n'}{r}=\frac{r'n'}{r}.
    \end{equation}
    Note also that the sets $A_1', \ldots, A_{h+1}'$ still satisfy properties \ref{partition(ii)}--\ref{partition(vi)} from Lemma~\ref{l:4.3chr}, with slightly modified constants, as detailed below. Set \[G'=G[A_1'\dcup\dots\dcup A_{h+1}'],\]  and proceed to the next step.

    \smallskip

    {\bf Step 2.~Covering $A_{h+1}'$ with a $K_{r-hs}$-factor $\cP$}

    \smallskip
    
    We now expose the random graph $G_3$ in order to find an appropriate $K_{r-hs}$-factor in $(G'\cup G_3)[A_{h+1}']$. Moreover, since we later want to extend this factor to a $K_r$-factor, we need to carefully treat the vertices with many non-neighbours outside $A_{h+1}'$. Thus, let 
    \[Z' = \{v \in A_{h+1}': |A_j' \setminus N_G(v)| \geq \frac{\beta}{10r}|A_j'| \text{ for some } j\in[h]\}.\] 
    Using property \ref{partition(vi)} from Lemma~\ref{l:4.3chr}, we have 
    \[ |Z'| \leq 20hr \beta^{-1} \gamma_h |A_{h+1}'| .\] 

    Consider first the case $h=m$. Note that $r-hs=t \leq s$. Moreover, for any vertex $v\in Z'$, there exists $j\in[h]$, such that 
    \[ \text{deg}_{G'[A_{h+1}']}(v) \geq \frac{\beta}{20r}|A_j'| \geq \frac{\beta}{20rs}|A_{h+1}'|.\] 
    Hence, using Lemma~\ref{l:Z-conf.-Kt-factor} applied to the graph $G'[A_{h+1}']$ with $\beta_{\ref{l:Z-conf.-Kt-factor}}=\frac{\beta}{20rs}$ and $\delta_{\ref{l:Z-conf.-Kt-factor}}=20mr\beta^{-1}\gamma_h$, by~\eqref{eq:gamma-h} w.h.p.~we can find in $(G'\cup G_3)[A_{h+1}']$ a $Z'$-conforming $K_t$-factor $\cP$.

    \smallskip
    
    Now consider the case $0<h<m$ and recall by \eqref{eq:r'} that $r-hs = r'$ and $|A_{h+1}|=\tilde{n}+g_{h+1}$ with $\tilde{n}=\frac{r'}{r}n$. Then by \eqref{eq:A_{h+1}'} and \eqref{eq:g_i-sizes}, we have
    \begin{align*} 
        ||A_{h+1}'|-\tilde{n}| = |A_{h+1}'|\left|1- \frac{n}{n'}\right| \leq \frac{2rg}{n}|A_{h+1}'| \leq 2rh\gamma_h |A_{h+1}'|.
    \end{align*}
     Using Claim~\ref{c:min-degree-in-A_h+1} applied with $A=A_{h+1}'$, we get
    \[\delta(G'[A_{h+1}'] \geq \left(1-\frac{s}{r'}-2rh\gamma_h\right)|A_{h+1}'|\] 
    and any set $X\subset A_{h+1}'$ with $|X|=\frac{s|A_{h+1}'|}{r'}$ satisfies 
    \[e_{G'}(X) \geq \frac12\gamma_{h+1}^2n^2 > \gamma_{h+1}^2|A_{h+1}'|.\] 
    Hence, using Theorem~\ref{t:non-extremal-conforming} with $r_{\ref{t:non-extremal-conforming}}=r'$, $Z_{\ref{t:non-extremal-conforming}}=Z'$, $\zeta_{\ref{t:non-extremal-conforming}}=20hr\beta^{-1}\gamma_h$ and $\gamma_{\ref{t:non-extremal-conforming}} = \gamma_{h+1}^2$, we conclude that w.h.p.~the graph $(G \cup G_3)[A_{h+1}']$ contains a $Z'$-conforming $K_{r'}$-factor $\mathcal{P}$.

   \smallskip

    {\bf Step 3.~Extending $\cP$ to a $K_{r}$-factor}

    \smallskip

    Now form the graph $\tilde G = G'/\mathcal{P}$ by \textit{collapsing} each clique $B \in \mathcal{P}$ into a single vertex; that is, the vertex set $A_{h+1}'$ is replaced by the set $\tilde A_{h+1}$ which is in one-to-one correspondence with the collection $\mathcal{P}$, and a vertex $v_B \in \tilde A_{h+1}$ representing the clique $B\in\cP$ is adjacent in $\tilde{G}$ to $v \in A_1'\dcup\dots\dcup A_h'$ if and only if $V(B) \subset N_{G'}(v)$. We will apply Lemma~\ref{l:st-reg-pairs-factors} to the graph $\tilde G$ after verifying the super-regularity conditions. Before we proceed, set
    \[ \tilde r = sh+1 \quad \text{ and } \quad \tilde{n}=|V(\tilde G)|,\]
    and note that
    \[ \tilde n = \frac{sh+1}{r}n' = \frac{\tilde r}{r}n'. \]
    In particular,
    \[ |A_{i}'| = \frac{sn'}{r}=\frac{s\tilde n}{\tilde r} \text{ for } i\in[h] \quad \text{ and } \quad |\tilde A_{h+1}|=\frac{n'}{r}=\frac{\tilde{n}}{\tilde r}.\]

    To verify the super-regularity conditions, notice first that by properties \ref{partition(v)} and \ref{partition(vi)} of Lemma~\ref{l:4.3chr}, each pair $(A_i,A_j)$, $i,j\in[h]$, is $(\gamma_h^{1/3} , d, 1/3)$-super-regular for some $d\geq 1-\gamma_h > 1/2$. Passing to subsets $A_i'\subset A_i$, and using Lemma~\ref{l:slicing}, we get that the pairs $(A_i', A_j')$, $i,j\in[h]$, are $(2\gamma_h^{1/3} , d', 1/4)$-super-regular for some $d'\geq 1/2$. 
    
    Consider now the pair $(A_i', \tilde A_{h+1})$, $i\in[h]$. We first check the degree condition for super-regularity. Take any vertex $v_B\in \tilde A_{h+1}$. The clique $B$ contains at most one vertex from $Z'$, and by property \ref{partition(iii)} of Lemma~\ref{l:4.3chr} such a vertex has at least $\frac12\beta n$ neighbours in $A_i'$. On the other hand, by the definition of $Z'$, the remaining vertices in $B$ have each at most $\frac{\beta}{10r}|A_i'|$ non-neighbours in $A_i'$. Thus, $v_B$ has at least \[\frac12\beta n - \frac{\beta(r-sh)}{10r}|A_i'| \geq \frac{\beta}{3}|A_i'|\]
    neighbours in $A_i'$. On the other hand, property \ref{partition(ii)} of Lemma~\ref{l:4.3chr} yields that any vertex $x\in A_i'$ has at most $4\beta n$ non-neighbours in $A_{h+1}'$ and the same is true for $\tilde A_{h+1}$.

    As for the regularity, by property \ref{partition(vi)} of Lemma~\ref{l:4.3chr} there are at most \[2\gamma_h|A_i'||A_{h+1}'|=2(r-sh)\gamma_h|A_i'||\tilde A_{h+1}|\] missing edges between $A_i'$ and $A_{h+1}'$, and this is also an upper bound for the number of missing edges between $A_i'$ and $\tilde A_{h+1}$. We conclude that the pair $(A_i', \tilde A_{h+1})$ is $((2r\gamma_h)^{1/3}, d'', \beta/3)$-super-regular for some $d'' \geq 1-2r\gamma_h \geq 1/2$.

    Finally, we expose the random graph $G_4$. By Lemma~\ref{l:st-reg-pairs-factors}, the graph $\tilde G\cup G_4[A_1'\dcup\dots\dcup A_h'\dcup \tilde A_{h+1}]$ has a $K_{hs+1}$-factor in which each clique uses exactly $s$ vertices from any set $A_i'$, $i\in[h]$, and one vertex from $\tilde A_{h+1}$.  This factor corresponds in turn to a $K_r$-factor in $G \cup G_3 \cup G_4\left[\bigcup_{i \in [h+1]} A_{i}'\right]$. Together with the $K_r$-cliques in the collection $\cC$ constructed in Step 1.~of the proof, we obtain the desired $K_r$-factor in $G\cup G_1\cup \dots \cup G_4 \subset G\cup \Gnp$.

    \end{proof}

    \section{The main extremal ingredient}

     In this section, we introduce a family of auxiliary graphs which, together with a suitable weighted graph packing, will play crucial role in splitting the parts of the regular partition in the non-extremal case into smaller parts of appropriate sizes and preserving important properties of the regular pairs. To get the desired sub-partition, we will also need to use some carefully constructed absorbers. 
	
    \subsection{The auxiliary graph $Q_h^{(m)}$}

    \begin{definition}
        For every $h\in [m]$ define the graph $Q_h^{(m)}$ that consists of disjoint 
        \begin{itemize}
            \item $h$-cliques $L_1,\dots,L_{s-t},N_1,\dots N_{(m-h)s+t}$ and
            \item $(m-h+1)$-cliques $M_1,\dots,M_{(m-h)s+t}$
        \end{itemize}
        together with all edges from $L_i$ to $M_j$ and all edges from $M_j$ to $N_j$, for any $i\in [s-t]$, $j\in[(m-h)s+t]$. 
    \end{definition}

    \begin{remark}\label{rem:Q_h}
        Note that $|V(Q_h)|=(m+1-h)r$. Moreover, in case $t=0$, $Q_m^{(m)}$ is just a~union of $s$ vertex-disjoint copies of $K_m$.
    \end{remark}

    When $m$ is clear from the context, we will use the notation $Q_h\coloneq Q_h^{(m)}$. We will also refer to the cliques $L_1,\dots, L_{s-t}$ as the $L$-sets, or to the $L$-part of $Q_h$. Similarly, we will refer to the $M$- and $N$-sets, or $M$- and $N$-part.

   \begin{figure}
    \begin{tikzpicture}[scale=0.8]
        \centering

        \coordinate (v1) at (0,2);
        \coordinate (v2) at (0,-2);
        \coordinate (v3) at (4,3);
        \coordinate (v4) at (4,1);
        \coordinate (v5) at (4,-3);
        \coordinate (v6) at (8,3);
        \coordinate (v7) at (8,1);
        \coordinate (v8) at (8,-3);
        
        \draw[line width=3] (v3) -- (v6);
        \draw[line width=3] (v4) -- (v7);
        \draw[line width=3] (v5) -- (v8);
        \draw[line width=3] (v1) -- (v3);
        \draw[line width=3] (v1) -- (v4);
        \draw[line width=3] (v1) -- (v5);
        \draw[line width=3] (v2) -- (v3);
        \draw[line width=3] (v2) -- (v4);
        \draw[line width=3] (v2) -- (v5);
        
        \draw[fill=white] (0,2) circle [radius=0.5];
        \draw[fill=white] (0,-2) circle [radius=0.5];

        \draw[fill=white] (4,3) circle [radius=0.7];
        \draw[fill=white] (4,1) circle [radius=0.7];
        \draw[fill=white] (4,-3) circle [radius=0.7];
        
        \draw[fill=white] (8,3) circle [radius=0.5];
        \draw[fill=white] (8,1) circle [radius=0.5];
        \draw[fill=white] (8,-3) circle [radius=0.5];
        
        \node at (0,0) {$\vdots$};
        \node at (4,-1) {$\vdots$};
        \node at (8,-1) {$\vdots$};

        \node at (0,4) [above] {$K_h$-copies};
        \node at (4,4) [above] {$K_{m-h+1}$-copies};
        \node at (8,4) [above] {$K_h$-copies};
           
        \node at (0,-4) [below] {$L$-sets};
        \node at (4,-4) [below] {$M$-sets};
        \node at (8,-4) [below] {$N$-sets};

        \draw [decorate,decoration={brace,amplitude=7pt,mirror,raise=4ex}]
  (-0.5,2.5) -- (-0.5,-2.5) node[midway,xshift=-4em]{$s-t$};
        \draw [decorate,decoration={brace,amplitude=7pt,mirror,raise=4ex}]
  (8.5,-3.5) -- (8.5,3.5) node[midway,xshift=6em]{$(m-h)s+t$};

    \end{tikzpicture}
\caption{The graph $Q_h^{(m)}$. Thick edges represent complete bipartite graphs.}
\end{figure}
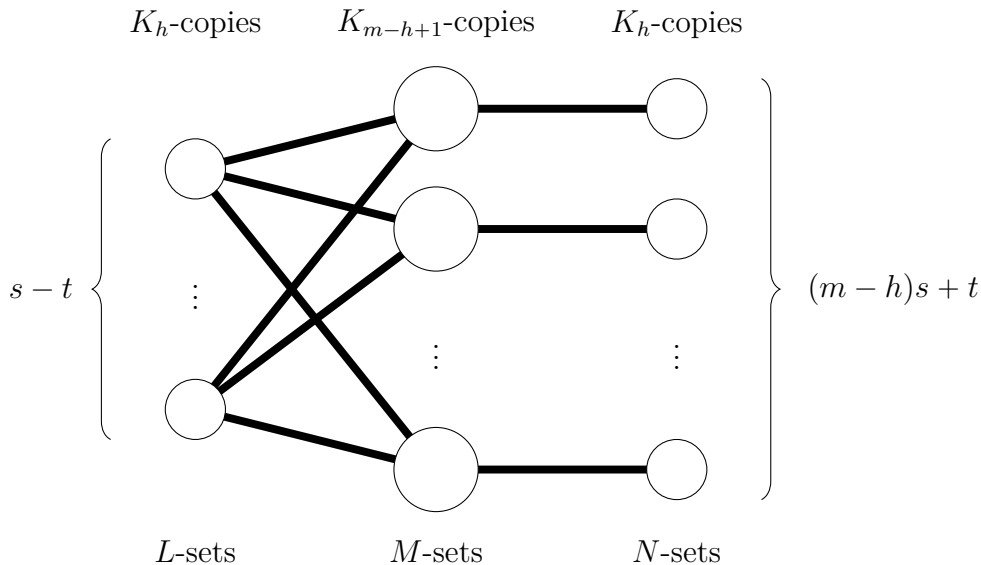

    \begin{lemma}\label{l:6.2chr}
    Let $r = ms+t$ with $0 \leq t <s$. There exists a constant $C>0$ such that for all $\gamma>0$ the following holds. There exists a constant $\delta >0$ and $n_0\in\mathbb{N}$ such that every graph $G$ on $n$ vertices, $n\geq n_0$, with $\delta(G)\geq \left(1-\frac{s}{r}-\delta\right)n$, satisfies one of the following two conditions.
    \begin{enumerate}[label=(\roman*)]
        \item There are vertex-disjoint copies of $K_{m+1},Q_1^{(m)},\dots, Q_m^{(m)}$ in $G$ covering all but at most $C$ vertices, or
        \label{l:6.2chr(i)}
        \item $G$ has an independent set of size $\left(\frac{s}{r}-\gamma\right)n$. \label{l:6.2chr(ii)}
        \end{enumerate}
    \end{lemma}

    Note that by a simple reparametrization of $m$, we immediately get the following corollary which will be useful in the singular case $t=s$.
    
    \begin{cor}\label{cor:6.2chr}
    Let $r = (m+1)s$. There exists a constant $C>0$ such that for all $\gamma>0$ the following holds. There exists a constant $\delta >0$ and $n_0\in\mathbb{N}$ such that every graph $G$ on $n$ vertices, $n\geq n_0$, with $\delta(G)\geq \left(1-\frac{1}{m+1}-\delta\right)n$, satisfies one of the following two conditions.
    
    \begin{enumerate}[label=(\roman*)]
        \item There are vertex-disjoint copies of $K_{m+2},Q_1^{(m+1)},\dots, Q_{m+1}^{(m+1)}$ in $G$ covering all but at most $C$ vertices, or
        \item $G$ has an independent set of size $\left(\frac{1}{m+1}-\gamma\right)n$.
        \end{enumerate}
    \end{cor}
    
    \begin{proof}[Proof of Lemma~\ref{l:6.2chr}]
         Let $C>r^4$ be sufficiently large for the proofs of Claim~\ref{c:p-A-edges-in-m+1} and Claim~\ref{c:p-B-edges} (more specifically, the contained applications of the K\H{o}vari--S\'os--Tur\'an Theorem) to hold, and let $\gamma >0$. For the remaining constants ($\beta, \xi, \eta, \delta$), we have explicit restrictions 
          \[ \beta<r^{-(r+1)}, \quad \xi < \min \{ r^{-3}, \gamma \},\quad 2\eta \leq \xi^2 \beta, \quad \delta \leq \min\left\{\frac{\xi-2\xi^2}{m+1}, \eta \right\}.\]
        Finally, let $n_0$ be sufficiently large.
        Let $G$ be as in the statement of the Lemma and set \[\alpha=1-\frac{s}{r}=\frac{(m-1)s+t}{ms+t},\] so that $\delta(G)\geq (\alpha-\delta)n$. Let \[\Pc=\{K_1,\dots,K_{m+1},Q_1,\dots, Q_m\}.\]
        Note that evidently $G$ has at least one $\Pc$-factor (for instance, each vertex can be covered by its own copy of $K_1$). Given a $\Pc$-factor $\cF$ of $G$, we let $k_i^\cF$ be the number of copies of $K_i$ in $\cF$ (for $i\in [m+1]$) and $q_h^\cF$ be the number of copies of $Q_h$ in $\cF$ (for $h\in [m]$). Now, set
        $$\phi_h^\cF=k_h^\cF+(s-t)q_h^\cF \quad \text
            {for all }\ h\in [m],$$
        $$\phi_{m+1}^\cF=k_{m+1}^\cF+\sum_{h=1}^m((m-h)s+t)q_h^\cF.$$
        We call the vector
    \[(\phi_{m+1}^\cF,\phi_m^\cF,q_m^\cF,\dots, \phi_1^\cF,q_1^\cF)\]
        the \emph{index} of the factor $\cF$. Fix a $\cP$-factor $\cF$ whose index is lexicographically maximal. 
        For $i \in [m+1]$, let $\fA_i$ be the set of $K_i$-copies in $\cF$, and let $A_i$ be the set of vertices defined by them. Define $\fB_h$ and $B_h$ similarly with respect to $Q_h$, for $h\in [m]$. 
        
        Now, if $\sum_{i=1}^m|A_i| \leq C$, then the conclusion \ref{l:6.2chr(i)} holds. So suppose that this is not the case and let $j\in [m]$ be minimal such that 
        \begin{equation}\label{eq:A_j}
            |A_j|\geq \frac{C}{m}.
        \end{equation}
        Our task is to show that $G$ has an independent set of size roughly $\frac{s}{r}n$. We will show that such a set exists by first proving a number of auxiliary claims concerning the properties of factor $\cF$.

        \begin{claim}
            \label{c:singular-A_m}
            If $t=0$, then $|A_m|< sm$ and hence $j<m$.
        \end{claim}
        
        \begin{proof}
            Assume the opposite, so $\cF$ contains at least $s$ copies of $K_m$. By Remark~\ref{rem:Q_h}, we can replace those $s$ copies of $K_m$ by a single copy of $Q_m$. This increases $q_m^\cF$ without changing $\varphi_{m+1}^{\cF}$ and $\varphi_m^{\cF}$.  Hence, we get a new $\Pc$-factor $\cF'$ with an index larger than that of $\cF$. This contradicts the maximality of $\cF$.  
        \end{proof}

        For vertex sets $X,Y\subset V(G)$, let $e(X,Y)$ denote the number of ordered pairs $(u,v)$ with $u\in X$ and $v\in Y$. In particular, if $X\cap Y=\emptyset$, then $e(X,Y)$ is just the number of edges in $G$ with one endpoint in $X$ and the other in $Y$. On the other hand, $e(X, X) = 2e(X)$, which will be used in the following claim and throughout the proof.
        
        \begin{claim}
            \label{c:p-A-edges-in-[p,m]}
            For any $h\in [j,m]$, if $h<m$ or $t>0$, then 
            $$e(A_j,A_h)\leq (\alpha-\beta)|A_j||A_h|.$$
        \end{claim}
       
        \begin{proof}
            Assume that the edge density between $A_j$ and $A_h$ is greater than $\alpha-\beta$.
            Then there exist $\tilde K_j\in \fA_j$ and $\tilde K_h\in \fA_h$ such that $$e(V(\tilde K_j),V(\tilde K_h))>(\alpha-\beta)jh.$$ 
            Note that
            $\alpha jh - j(h-1) = j\left(h-\frac{sh}{r}-h+1 \right)=j\left(\frac{(m-h)s+t}{ms+t}\right)>0$, since we assumed that $h<m$ or $t>0$. Hence,  $(\alpha-\beta) jh > j(h-1)$ for sufficiently small $\beta$. It follows that $$e(V(\tilde K_j),V(\tilde K_h))\geq j(h-1)+1.$$ By the pigeonhole principle, there is a vertex $x\in \tilde K_j$ such that $x$ is a neighbour of all vertices in $\tilde K_h$. Replace $\tilde K_j$ and $\tilde K_h$ with two cliques $ K_{j-1}'$ and $ K_{h+1}'$, by shifting the vertex $x$ from $\tilde K_j$ to $\tilde K_h$. This increases $\phi_{h+1}^\cF$ while the preceding terms stay intact, so again we get a contradiction to the maximality of $\cF$. 
        \end{proof}

        \begin{claim}
            \label{c:p-A-edges-in-m+1}
            $$e(A_j,A_{m+1})\leq |A_j|\left((\alpha-\beta)|A_{m+1}|+C\right).$$
        \end{claim}

        \begin{proof}
            Assume that this is not the case. Then 
            \begin{equation}\label{eq:A_{m+1}}
                |A_{m+1}|\geq C,
            \end{equation}
            as otherwise $e(A_j,A_{m+1})\leq|A_j|C$, and
            \begin{equation}
            \label{eq:edges-a}
                e(A_j,A_{m+1})> (\alpha-\beta) |A_j||A_{m+1}|.
            \end{equation}
            Construct an auxiliary bipartite graph $H$ with vertex classes $\fA_j$, $\fA_{m+1}$ by joining $\tilde K_j$ to $\tilde K_{m+1}$ if the edge density between those cliques is greater than $\alpha-2\beta$, that is if
            \[e(V(\tilde K_j), V(\tilde K_{m+1}))> (\alpha-2\beta)j(m+1).\]
            Owing to \eqref{eq:edges-a}, by standard double counting, we have 
            \[e(H)\geq \beta |\fA_j||\fA_{m+1}|.\]
            Consider an edge $f=\tilde K_j \tilde K_{m+1}$ of $H$. We claim that 
            \begin{equation}
            \label{eq:edges-K_j_m+1}
            e(V(\tilde K_j), V(\tilde K_{m+1}))\geq mj.
            \end{equation} 
            Indeed, consider the following two cases.
            
            \textbf{Case 1.} If $t>0$, then recalling that $\alpha=1-\frac{s}{r}>1-\frac1m$, we get
            \[\alpha (m+1)j > \frac{(m-1)(m+1)}{m}j=mj-\frac{j}{m}\geq mj-1.\]
            
            \textbf{Case 2.} If $t=0$, then $\alpha=1-\frac1m$, and by Claim~\ref{c:singular-A_m} we have $j<m$. Then, similarly, \[\alpha (m+1)j = \frac{(m-1)(m+1)}{m}j=mj-\frac{j}{m}> mj-1.\]
            In both cases, for sufficiently small $\beta$, we have $(\alpha-2\beta)j(m+1) > mj-1$, which implies~\eqref{eq:edges-K_j_m+1}.
            In other words, there are at most $j$ missing edges between $\tilde K_j$ and $\tilde K_{m+1}$. Hence, there is a partition $V(\tilde K_{m+1})=M_f\cup N_f$ such that $|M_f|=m+1-j$, $|N_f|=j$ and $V(\tilde K_j)\cup M_f$ span a copy of $K_{m+1}$ in $G$. 
            
            Next, notice that by our assumption on $\beta$, we have $\beta<\binom{m+1}{j}^{-1}$. Then since the number of possible partitions is $\binom{m+1}{j}$, for each $\tilde K_{m+1}\in\fA_{m+1}$ we can choose a partition $V(\tilde K_{m+1})=\tilde M \cup \tilde N$ such that there are at least $\beta \deg_H(\tilde K_{m+1})$ edges $f$ in graph $H$ with $M_f=\tilde M$ and $N_f=\tilde N$. Let $H'$ be the subgraph of $H$ consisting of such edges. We have $$e(H')\geq \beta e(H)\geq \beta^2 |\fA_j||\fA_{m+1}|.$$
            Since, by \eqref{eq:A_j}, $|\fA_j|\geq \frac{C}{m^2}$ and, by \eqref{eq:A_{m+1}},  $|\fA_{m+1}|\geq \frac{C}{m+1}$, the theorem of K\H{o}vari, S\'os and Tur\'an~\cite{kst54}\footnote{This application of the K\H{o}vari, S\'os and Tur\'an Theorem to a graph of density $\beta^2$, as well as the subsequent one, are our main restrictions on $C$.} yields a copy of $K_{s-t,(m-j)s+t}$ in $H'$. We can see that the vertices of this $K_{s-t,(m-j)s+t}$-copy form, in turn, a copy of $Q_j$ in $G$ (the $s-t$ copies of $K_j$ form the $L$-part of $Q_j$, while for each copy $\tilde K_{m+1}$ we take the aforementioned partition $\tilde M \cup\tilde N$ which corresponds to the $M$-part and $N$-part). By replacing the $s-t$ copies of $K_j$ and $(m-j)s+t$ copies of $K_{m+1}$ with $Q_j$ in $\cF$, we can increase $q_j^\cF$ without changing $\phi_{m+1}^\cF$ or $\phi_j^\cF$, contradicting the maximality of $\cF$. 
        \end{proof}

        \begin{claim}
            \label{c:p-B-edges}
            If $h\in[j-1]$, then 
            $$e(A_j,B_h)\leq |A_j|((\alpha-\beta)|B_h|+C).$$
        \end{claim}
        
        \begin{proof}
            The proof is similar to the proof of the previous claim. Again we assume for the sake of contradiction that 
            \begin{equation}\label{eq:B_h}
                |B_h|\geq C \quad \text{ and } \quad e(A_j,B_h)> (\alpha-\beta)|A_j||B_h|.
            \end{equation}
            Like in the proof of Claim~\ref{c:p-A-edges-in-m+1}, we construct an auxiliary bipartite graph $H$ between $\fA_j$ and $\fB_h$ by joining $\tilde K_j$ and $\tilde Q_h$ if the edge density between them is greater than $\alpha-2\beta$, that is, if 
            \[e(V(\tilde K_j), V(\tilde Q_h))> (\alpha-2\beta)j|V(\tilde Q_h)|.\]
            Notice that by the same reasoning as before, we have
            \[ e(H) \geq \beta|\fA_j||\fB_h|.\]
            Recall that $|V(Q_h)| = (m+1-h)r$. Fix an edge $f=\tilde K_j \tilde Q_h$ of $H$. 
            Let 
            \[
            X_f=\{x\in V(\tilde Q_h)\mid V(\tilde K_j) \subseteq N(x)\}.
            \]
            Note that 
            \[|X_f|  \geq e(V(\tilde K_j), V(\tilde Q_h))-(j-1)|V(\tilde Q_h)|\geq ((\alpha-2\beta)j-(j-1))(m+1-h)r.\] 
            We claim that, for sufficiently small $\beta$, the right-hand side in the second inequality above is larger than $(m-j)((m-h)s+t).$ Indeed, we have
            $$(\alpha j -j+1)(m+1-h)(ms+t)-(m-j)(ms-hs+t)= s(m-j)+t(j+1-h)>0,$$
            since we assumed that $h \leq j-1$, and since by Claim~\ref{c:singular-A_m} $j<m$ in case $t=0$.
            We conclude that
            \begin{equation}
             \label{eq:size-Xf}
            \begin{split}
            |X_f| 
            & > (m-j)((m-h)s+t).
            \end{split}
            \end{equation}

            If some vertex $x\in X_f$ is contained in the $L$- or $N$-set of $\tilde Q_h$, we can replace $\tilde K_j$ and $\tilde Q_h$ with a copy of $K_{j+1}$ (by shifting the vertex $x$ to $\tilde K_j$), $(m-h)s+t$ copies of $K_{m+1}$, one copy of $K_{h-1}$, and $s-t-1$ additional copies of $K_h$. This increases $\phi_{j+1}^\cF$ without decreasing $\phi_{m+1}^\cF$, which contradicts the maximality of $\cF$. For this reason $X_f$ is contained in the union of the $M$-sets of $\tilde Q_h$. 
            
            Now, by \eqref{eq:size-Xf} there is some index $i(f)\in [(m-h)s+t]$ such that $\tilde M_{i(f)}$ intersects $X_f$ in at least $m-j+1$ vertices. This means that there is a partition 
            \[\tilde M_{i(f)}=C_f\cup D_f\]
            with $|C_f|=m-j+1$ and $|D_f|=j-h$, and such that $C_f\subset X_f$. Since for each $\tilde Q_h$ the number of possible pairs $(i(f), C_f)$ as above is at most $((m-h)s+t){m+1-h\choose j-h} < \beta^{-1}$, we can choose one, say $(i(\tilde Q_h), C(\tilde Q_h))$, for which there are at least $\beta \deg_H(\tilde Q_h)$ edges $f$ with $(i(f), C_f) = (i(\tilde Q_h),C(\tilde Q_h))$.
            Thus, like in the proof of the previous claim, there is a~subgraph $H'$ of $H$ satisfying
            \begin{itemize}
                \item $e(H')\geq \beta e(H) \geq  \beta ^2 |\fA_j| |\fB_h|$,
                \item for any $\tilde Q_h$, there is a "special pair" $(\tilde M_i, \tilde N_i)$ in $\tilde Q_h$ for which there is a unique $\tilde C=C(\tilde Q_h)\subseteq \tilde M_i$, such that $(i,\tilde C)=(i(f),C_f)$ for any edge $f=\tilde K_j \tilde Q_h$ in $H'$. 
            \end{itemize}

            By \eqref{eq:A_j} we have $|\fA_j|\geq \frac{C}{m^2}$, and by \eqref{eq:B_h}, $|\fB_h|\geq \frac{|B_h|}{|Q_h|}\geq\frac{C}{(m-h)s+t}$. The K\H{o}vari-S\'os-Tur\'an theorem implies now that $H'$ contains a subgraph isomorphic to $K_{s-t,(m-j)s+t}$. 
            We can take this subgraph, which consists of $s-t$ copies of $K_j$ and $(m-j)s+t$ copies of $Q_h$, and replace them with one copy of $Q_j$, $((m-j)s+t)((m-h)s+t-1)$ copies of $K_{m+1}$ and some additional copies of $K_h$ as follows:
            \begin{itemize}
                \item the $s-t$ copies of $K_j$ become the $L$-part of $Q_j$,
                \item we take the $(m-j)s+t$ special pairs $(\tilde M_i, \tilde N_i)$, each from a different copy of $\tilde Q_h$, and split $\tilde M_i$ into $\tilde M_i = \tilde C \cup \tilde D$, where $\tilde C = C(\tilde Q_h)$; the sets $\tilde C$ become the $M$-part of $Q_j$ and the sets $\tilde D \cup \tilde N$ become the $N$-part of $Q_j$, 
                
                \item the remaining pairs $(\tilde M_k, \tilde N_k)$ are just copies of $K_{m+1}$, 
                \item finally, the $L$-part of $\tilde Q_h$ gives the additional copies of $K_h$.
            \end{itemize} 
            
            This increases $q_j^\cF$, while $\phi_{m+1}^\cF$ and $\phi_{j}^\cF$ remain unchanged, which is again a contradiction to the maximality of $\cF$.
        \end{proof}

        \begin{claim}
            \label{c:p-B-non-neighbours-in-L-N}
            If $x\in A_j$, $h\in[j,m]$, and $\tilde Q_h\in \fB_h$, then $x$ has at least one non-neighbour in each $L$- and $N$-set of $\tilde Q_h$. In particular,
            $$|N_G(x)\cap V(\tilde Q_h)|\leq \alpha|V(\tilde Q_h)|.$$
        \end{claim}
        
        \begin{proof}
            Let $\tilde K_j\in \fA_j$ be the copy of $K_j$ that contains $x$. If $x$ is completely joined to any $L$- or $N$-set of $\tilde Q_h$, we can replace $\tilde Q_h$ and $\tilde K_j$ in $\cF$ with one copy of $K_{h+1}$ (containing $x$), $(m-h)s+t$ copies of $K_{m+1}$, one copy of $K_{j-1}$ ($\tilde K_j - x$) and some additional copies of $K_h$. This increases $\phi_{h+1}^\cF$ without changing $\phi_{m+1}^\cF$, contradicting the maximality of $\cF$.
            Additionally, recalling that $|V(\tilde Q_h)|=(m+1-h)r$, we get
            \begin{align*}|N_G(x)\cap V(\tilde Q_h)|&\leq |V(\tilde Q_h)|-((m-h)s+t)-(s-t) = \alpha |V(\tilde Q_h)|.
            \end{align*}
        \end{proof}

        Next, we define a partition $V(G)=X\cup Y$ by taking 
        \begin{align*}X =A_1\cup\dots\cup A_{m+1}\cup B_1\cup \dots \cup B_{j-1} \quad \text{ and } \quad 
        Y = B_j\cup \dots \cup B_m.\end{align*}
        
        \begin{claim}\label{c:X_is_not_too_large} We have
            $$|X|\leq \eta n.$$
        \end{claim}
        \begin{proof}
            Owing to the minimality of $j$, recall that $|A_h| \leq \frac{C}{m}$ for $h<j$. Then, using all the previous claims, we have
            \begin{align*}
                (\alpha-\delta)|A_j|n & \leq  e(A_j,V(G))\\
                &= \sum_{h=1}^{j-1}e(A_j,A_h)+\sum_{h=j}^{m+1}e(A_j,A_h)+\sum_{h=1}^{j-1} e(A_j, B_h)+\sum_{h=j}^m e(A_j, B_h) \\
                & \leq |A_j|\left(\frac{(j-1)C}{m} + (\alpha-\beta)|X|+jC +sm+\alpha|Y|\right)\\
                & \leq |A_j|\left((\alpha-\beta) |X|+\alpha (n-|X|)+(m+2)C\right),
            \end{align*}
            where the term $|A_j|sm$ is needed for case $t=0$ and comes from Claim~\ref{c:singular-A_m}.
            In particular,
            \[|X|\leq \frac{\delta n +(m+2)C}{\beta}\leq \frac{2\delta n}{\beta}\leq \eta n.\]
        \end{proof}

        Next, let $W$ be the union of the $L$- and $N$-sets of the copies of $Q_h$ in $\fB_j\cup\dots\cup\fB_m$. We construct an auxiliary bipartite graph $H$ between $A_j$ and $W$ with edges defined as follows. For a vertex $x\in A_j$ and a copy $\tilde Q_h\in \fB_h$, where $h\in [j,m]$:
        \begin{itemize}
            \item if $|N_G(x)\cap V(\tilde Q_h)|<\alpha |V(\tilde Q_h)|$, then $H$ has no edges from $x$ to $V(\tilde Q_h) \cap W$;
            \item otherwise, by Claim~\ref{c:p-B-non-neighbours-in-L-N} we have $|N_G(x)\cap V(\tilde Q_h)|=\alpha |V(\tilde Q_h)|=(m-h+1)s$ and there are exactly this many non-edges between $x$ and $V(\tilde Q_h)\cap W$, one for each $L$- and $N$-set. These $(m-h+1)s$ non-edges will become the edges in $H$ from $x$ to $V(\tilde Q_h)\cap W$. 
        \end{itemize}

        \begin{claim}\label{c:no-two-H-neighbours}
            No vertex $y\in W$ has two $H$-neighbours belonging to the same clique $\tilde K_j \in \fA_j$.
        \end{claim}
        \begin{proof}
            Assume such a vertex $y$ exists, and let $\tilde K_j\in \fA_j$ and $u,v \in V(\tilde K_j)$ be such that $uy,vy\in E(H)$. Let $\tilde Q_h \in \fB_h$ be the copy of $Q_h$ that contains $y$ and $\tilde Y$ be the corresponding $L$- or $N$-set of $\tilde Q_h$ with $y\in \tilde Y$. By the definition of $H$, both $u$ and $v$ are connected to every vertex in $\tilde Y \setminus \{y\}$ in $G$. Let us remove $y$ from $\tilde Y$ and add $u$ and $v$ instead. By doing so, we replace $\tilde K_j$ and $\tilde Q_h$ by a copy of $K_{h+1}$ (spanned by vertices in $\tilde Y\cup\{u,v\}\setminus \{y\}$), $(m-h)s+t$ copies of $K_{m+1}$, along with a copy of $K_{j-2}$ (i.e.~$\tilde K_j - \{u,v\}$) and $K_1$ (i.e.~vertex $y$), and additional copies of $K_h$. This increases $\phi_{h+1}^\cF$, while not decreasing $\phi_{m+1}^\cF$, thus contradicting the maximality of $\cF$.
        \end{proof}
        
        Finally, set 
        \[ Z=\{z\in W\colon d_H(z)\geq j+1\}. \]

        \begin{claim}
            \label{c:Z-is-large}
            We have $|Z|\geq j(1-\xi)(1-\alpha)n$.
        \end{claim}
        \begin{proof}
            Recall $|A_j|\geq \frac{C}{m}$, which implies $|\fA_j|\geq \frac{C}{jm}\geq \frac{C}{m^2}\geq j(m+1-j)+1$, for an appropriate choice of $C$. Pick a subset $\fA_j'\subseteq \fA_j$ of size $|\fA_j'|=j(m+1-j)+1$ and let $A_j'$ be the union of the vertex sets of the cliques in $\fA_j'$. Look at any vertex $x\in A_j'$. Let $\fD_x$ be the set of all $\tilde Q_h\in \fB_h$ such that $h\in [j,m]$ and $|N_G(x)\cap V(\tilde Q_h)|=\alpha |V(\tilde Q_h)|$.
            By Claim~\ref{c:p-B-non-neighbours-in-L-N}, all other $\tilde Q_h\in \fB_j\cup \dots \cup \fB_m$ satisfy $|N_G(x)\cap V(\tilde Q_h)|<\alpha|V(\tilde Q_h)|\leq (\alpha- \beta)|V(\tilde Q_h)|$ for $\beta$ sufficiently small. Let $D_x$ be the union of the vertex sets of the cliques in $\fD_x$. We have 
            \begin{align*}(\alpha-\delta)n\leq |N_G(x)| &\leq |X|+|N_G(x)\cap D_x|+|N_G(x)\cap Y\setminus D_x|\\
            &\leq \eta n + \alpha |D_x|+(\alpha-\beta)(n-|D_x|),
            \end{align*}
            whence 
            \[|D_x|\geq \frac{\beta-\delta-\eta}{\beta}n\geq (1-\xi^2)n.\]
            Now set 
            \[D=\bigcap_{x\in A_j'}D_x \quad \text{ and } \quad \fD=\bigcap_{x\in A_j'}\fD_x .\]
            Since, for $\xi$ sufficiently large, \[|D|\geq (1-|A_j'|\xi^2)n\geq (1-\xi)n,\]
            it suffices to argue that every $\tilde Q_h\in \fD$ has at least \[j(1-\alpha)|V(\tilde Q_h)|=(m+1-h)sj\] vertices belonging to $Z$. Fix $\tilde Q_h\in \fD$ and let $U$ be one of its $L$- or $N$-sets. We will show that $|U\cap Z|\geq j$, which implies the desired bound on $|V(\tilde Q_h)\cap Z|$. Suppose that $|U\cap Z|\leq j-1$. Since $|U|=h\geq j$, there is a set $U'$ of size $j-1$ such that $U \cap Z\subseteq U'\subseteq U$. Look at any $\tilde K_j \in \fA_j'$. Due to $\tilde Q_h\in \fD\subseteq \bigcap_{x\in V(\tilde K_j)}\fD_x$, every vertex of $\tilde K_j$ has an $H$-neighbour in $U$. Moreover, by Claim~\ref{c:no-two-H-neighbours}, $H$ induces a matching from $V(\tilde K_j)$ to $U$. Therefore, there is an $H$-edge from $\tilde K_j$ to $U\setminus U'$. Applying this argument to every $\tilde K_j \in \fA_j'$ we obtain $j(m+1-j)+1$ edges from $A_j'$ to $U\setminus U'$. But $|U\setminus U'|=h-(j-1)\leq m+1-j$, so there is a vertex $u\in U\setminus U'$ which receives at least $j+1$ of those edges. This implies that $u\in U\cap Z \subseteq U'$, which is a contradiction.
        \end{proof}

        \begin{claim}
            \label{c:Z-has-high-mindegree}
            We have
            \[\delta(G[Z])\geq\frac{j-1-2\xi}{j}|Z|.\]
        \end{claim}
        \begin{proof}
            Using the previous claim and the inequality $\alpha =1-\frac{s}{r}\leq 1-\frac{1}{m+1}$, we can see that
            \begin{align*}
                \delta(G[Z]) - \frac{j-1-2\xi}{j}|Z| &\geq (\alpha-\delta)n-(n-|Z|)-\frac{j-1-2\xi}{j}|Z|\\
                &=\frac{1+2\xi}{j}|Z|-(1-\alpha+\delta)n\\
                &\geq \left((1-\alpha)((1+2\xi)(1-\xi)-1)-\delta\right)n\\
                &\geq \left(\frac{1}{m+1}(\xi-2\xi^2)-\delta\right)n\geq 0.
            \end{align*}
        \end{proof}

        \begin{claim}
            \label{c:no-K_(j+1)-in-Z}
            There are no copies of $K_{j+1}$ in $G[Z]$.
        \end{claim}
        \begin{proof}
        Assume $y_1,\dots, y_{j+1}$ is a clique in $G[Z]$. By the definition of $Z$, there exists a~matching $x_1y_1,\dots , x_{j+1}y_{j+1}$ in $H$. As before, we will perturb the factor $\cF$ to increase its index. For each $i\in [j+1]$, remove the vertex $x_i$ from its copy of $K_j$, and  replace $y_i$ by $x_i$ in the copy of $Q_h$ containing $y_i$. This leaves all the copies of $Q_h$ containing one of the vertices $y_i$, $i\in[j+1]$, intact. We can now add the clique $y_1,\dots, y_{j+1}$ to the factor. This increases $\phi_{j+1}^\cF$,  contrary to the maximality of $\cF$.
        \end{proof}

        We can finally finish the proof of Lemma~\ref{l:6.2chr}.
        If $j=1$, then Claim~\ref{c:Z-is-large} and Claim~\ref{c:no-K_(j+1)-in-Z} show that $Z$ is an independent set of size at least $(1-\xi)(1-\alpha)n=\frac{s}{r}(1-\xi)n>(\frac{s}{r}-\gamma)n$. If $j\geq 2$, then the Andr\'asfari-Erd\H{o}s-S\'os~\cite{aes74} theorem together with Claim~\ref{c:Z-has-high-mindegree} and Claim~\ref{c:no-K_(j+1)-in-Z} show that $G[Z]$ is $j$-partite. This in turn means that there is an independent set whose size, due to Claim~\ref{c:Z-is-large}, is at least $\frac{|Z|}{j}\geq (1-\xi)(1-\alpha)n$. In both cases, we see that \ref{l:6.2chr(ii)} holds.
    \end{proof}  
    
    \subsection{Weighted graphs and packing}

    Throughout this subsection, let us assume that $r=ms+t$ with $0\leq t<s$.

    Recall the definition of a weighted graph and an $\cF$-packing from the Introduction.
    Any graph $F$ can be viewed as a~weighted graph with $w_F \equiv 1$. 
    
    The \textit{residue} of a packing is defined as $|V(G)|- \sum_{i \in I}\sum_{v \in V(F_i)}w_{F_i}(v)$. An \emph{$\mathcal{F}$-factor} is an~$\mathcal{F}$-packing with residue 0. If $\cF$ consists of a single weighted graph $F$, we will write $F$-factor or $F$-packing for short.
    
    Let $T$ denote the weighted complete graph on vertices $\sigma_1,\dots,\sigma_m$ and $\tau$, with vertices $\sigma_i$ having each weight $\frac{s}{r}$ and vertex $\tau$ having weight $\frac{t}{r}$. In case $t=0$, one can think of $T$ as the clique $K_m$ with each vertex bearing weight $1/m$ (formally, we could also exclude $\tau$ from the definition of $T$ in this case). 
    
    We can easily see that the graph $K_{m+1}$ has a~$T$-factor consisting of $m+1$ copies of $T$. Indeed, it suffices to take $f_i(\tau)$ to be a~different vertex for each $F_i$, $i\in[m+1]$. Recall also that, in case $t=0$, the graph $Q_m^{(m)}$ consists of $s$ vertex-disjoint copies of $K_m$. This graph naturally has a $T$-factor, since we can map $\tau$ anywhere without enhancing the total weight of a vertex. It turns out that we can also factorize all other graphs $Q_h^{(m)}$, possibly first rescaling the weights of $T$ by an appropriate rational number.
    Recall that for a weighted graph $F$ and $\phi > 0 $, by $\phi\ltimes F$ we denote the weighted graph obtained by multiplying all the vertex weights of $F$ by $\phi$. The weight function is therefore given by $w_{\phi\ltimes F}(x)=\phi \cdot w_F(x)$ for every $x\in V(F)$. 

    \begin{lemma}\label{l:factoring-Qh}
        The clique $K_{m+1}$ has a $T$-factor. Furthermore, for every $h\in [m]$, there is a rational number $q_h\in \QQ_{>0}$ such that $Q_h^{(m)}$ has a $\{q_h\ltimes T\}$-factor.
    \end{lemma}
    
    \begin{proof}
    The first part has been already justified, and we have also shown that in the case $t=0$, the graph $Q_m^{(m)}$ has a $T$-factor.
    So consider the graph $Q_h=Q_h^{(m)}$ and assume that $t>0$ or $h<m$. Set \begin{equation*}
        x=s-t, \quad y=(m-h)s+t,
    \end{equation*} 
    \begin{equation*}
        q=q_h=\frac{r}{s(m-h+1)xy}, \quad q_1=x\cdot q,\quad q_2=y\cdot q.    
    \end{equation*} 
    Since $q_1\ltimes T$ and $q_2\ltimes T$ have a $q\ltimes T$-factor, it suffices to find a $\{q_1\ltimes T, q_2\ltimes T\}$-factor.
        Consider first the weighted graph $F$ with vertices
        \begin{itemize}
            \item $l_1,\dots,l_h$ of weights $\frac 1y$,
            \item $m_1,\dots, m_{m-h+1}$ and $n_1,\dots,n_h$ of weights $\frac 1x$,
        \end{itemize}
        and all edges except $l_in_j$, $i,j \in [h]$.
        We can factor $Q_h$ with $xy$ copies of $F$ by assigning one copy of $F$ to $L_i\cup M_j \cup N_j$ for each $i\in [s-t]$ and $j\in [(m-h)s+t]$. 
        Therefore, it suffices to factor $F$ with copies of $q_1\ltimes T$ and $q_2\ltimes T$. 
        
        For $i\in[m-h+1]$, let $T_i$ be the copy of $T$ together with a map $f_i$ mapping the vertices of $T_i$ to $l_1,\dots, l_h,m_1,\dots ,m_{m-h+1}$, in such a way that vertex corresponding to $\tau$ in $T_i$ is mapped to $m_i$. Similarly, for $i\in[m-h+1]$, let $T_i^*$ be the copy of $T$ with a map $f_i^*$ mapping the vertices of $T_i^*$  to $m_1,\dots ,m_{m-h+1},n_1,\dots ,n_h$ in such a way that vertex corresponding to $\tau$ in $T_i^*$ is again mapped to $m_i$. 
        Then $q_1\ltimes T_i$ and $q_2\ltimes T_i^*$, $i\in[m-h+1]$, factorise $F$. Indeed, each vertex $l_i$ gets the total weight
        \[ q_1(m-h+1)\frac{s}{r} = \frac1y,\]
        similarly, each vertex $n_i$ gets the total weight
        \[ q_2(m-h+1)\frac{s}{r} = \frac1x,\]
        and each vertex $m_i$ gets the total weight
        \begin{align*} 
        (q_1+q_2)\left(\frac{t}{r} + (m-h)\frac{s}{r}\right) & = \left(\frac{1}{x}+\frac{1}{y}\right) \frac{(m-h)s+t}{s(m-h+1)} \\
        & = \frac{1}{x}\left(1+ \frac{s-t}{(m-h)s+t}\right)\frac{(m-h)s+t}{s(m-h+1)} = \frac1x.
        \end{align*}
    \end{proof}
    
    \begin{cor}\label{cor:T_packing}
    For any $m\in\NN$, there exists $b\in \NN$ such that any collection of vertex-disjoint copies of $K_{m+1}, Q_1^{(m)}, \dots, Q_m^{(m)}$ has a $\left\{\frac 1b\ltimes T\right\}$-factor.
    \end{cor}

    \begin{proof}
        By Lemma~\ref{l:factoring-Qh}, any collection of vertex-disjoint copies of $K_{m+1},Q_1^{(m)},\dots,Q_m^{(m)}$ has a $\left\{T,q_1\ltimes T, \dots , q_m\ltimes T\right\}$-factor, for some $q_1,\dots, q_m\in \QQ_{>0}$. Let $b$ be the common denominator of $q_1,\dots , q_m$. For every $h\in [m]$, $q_h\ltimes T$ can be factorized using $b\cdot q_h$ copies of $\frac 1b\ltimes T$, so we can find a $\frac 1b \ltimes T$-factor for our collection.
    \end{proof}

    \subsection{Stability for Shokoufandeh--Zhao Theorem}
        \label{sec:sz-stability}
        In this section, we prove Theorem~\ref{t:sz-stability}. Let us first introduce the \emph{critical chromatic number} $\chi_{cr}$. Let $H$ be an $r$-vertex graph with chromatic number $m+1$, and let $t$ be the smallest class size in any proper $(m+1)$-colouring of $H$. Let $s$ be defined by $r = ms+t$, and note that this is not necessarily an integer. Then the critical chromatic number of $H$ is given by $\chi_{cr}(H) = \frac{r}{s}$.
        
        We will only use the following fact which follows directly from the definition and can be found in~\cite{sz03}. They refer to the graph $B(m, r, t)$ below as a \textit{bottle} graph.

        \begin{fact}
        Let $H$ be as above, and let $B(m,r,t)$ be the complete $(m+1)$-partite graph with $m$ classes each of size $ms$ and one class of size $mt$. Then $B(m,r,t)$ has an $H$-factor, and $\chi_{cr}(B(m,r,t)) = \frac rs$.
        \end{fact}

        We can now proceed with the proof.

        \begin{proof}
            [Proof of Theorem~\ref{t:sz-stability}]
            Let $H$ be the given graph, and let $r, m, t$ and $s$ be as above. As in the previous subsection, let $T$ denote the weighted complete graph on vertices $\sigma_1,\dots,\sigma_m$ and $\tau$, with $m$ vertices of weight $\frac{s}{r}$ and one vertex of weight $\frac{t}{r}$. Since $\chi_{cr}(H) = \frac{r}{s}$, the graph $G$ satisfies the hypothesis of Lemma~\ref{l:6.2chr}. Using Lemma~\ref{l:6.2chr} (conclusion (i)) and Corollary~\ref{cor:T_packing}, there is an integer $b \in \NN$ such that $G$ has a $\left\{ b^{-1} \ltimes T\right\}$-packing with remainder at most $C$. Furthermore, each copy of $T$ has a $\left\{b_1^{-1} \ltimes B(m,s,t)\right\}$-factor, for some integer $b_1$, which in turn yields a $\left\{b_2^{-1} \ltimes H\right\}$-factor for some integer $b_2$, as required.
        \end{proof}

    \subsection{Constructing the absorber}
        
    In this subsection, let $r=ms+t$ with $1\leq t \leq s$. This is a small change of parametrization compared to the previous subsections, as the case $t=s$ will now be parametrised as $r=(m+1)s$ instead of $r=ms$. 

    \begin{definition} 
        \label{def:absorber}
        A clique $K_g$ in $G$ is called \textit{good} for a vertex $v$ if $v$ has at least $g-2$ neighbours in $K_g$. A \textit{$(k, m, \delta)$-absorber} in $G$ is a collection $\absorber$ of vertex-disjoint copies of  $K_{m+k}$ in $G$ covering at most $10(k+m)\delta n$ vertices and such that for each pair of vertices $u,v\in V(G)$, $\absorber$ contains at least $\delta^2 n$ cliques which are good for both $u$ and $v$. 
    \end{definition}
    
    \begin{lemma}\label{l:absorber}
        Let $0<\delta<(4mr)^{-1}$ and suppose that $\delta(G)\geq \left(1-\frac sr -\delta\right)n$.
        \begin{enumerate}[label=(\roman*)]
            \item\label{l:absorber(i)} If $t < s$, then $G$ contains a~$(1, m, \delta_2)$-absorber for some constant $\delta_2=\delta_2(r)>0$. 
            
            \item\label{l:absorber(ii)} Suppose $t = s$, and that every subset $X\subset V(G)$ with $|X|\geq\left(\frac1{m+1}-m\delta\right)n$ satisfies $e_G(X)\geq\delta_1 n^2$ for some $\delta_1>0$. Then $G$ contains a~$(2, m, \delta_2)$-absorber for some constant $\delta_2=\delta_2(r,\delta_1)>0$.
            
        \end{enumerate}
    \end{lemma}
        
    \begin{proof}
        For both cases we will need two auxiliary claims.
        \begin{claim}\label{c:absorber_abstract_claim}
            Let $\ell = n^{O(1)}$, and let $\cK_1, \ldots, \cK_\ell$, be collections of $K_{g}$-cliques in a graph $G$ satisfying $|\cK_i| \geq \xi n^{g}$. Then there is a collection $\absorber$ of vertex-disjoint copies of $K_g$ covering at most $\xi g n$ vertices such that 
            $$|\cA \cap \cK_i| \geq \xi^2 n/50$$
            for $i \in [\ell]$.
        \end{claim}

        \begin{proof}            
            Suppose without loss of generality that $|\cK_i|=\xi n^{g}$ for each $i\in[\ell]$. Let $\cK=\bigcup_i \cK_i$, and note that $|\cK|\leq n^g$ since $\cK$ consists only of copies of $K_g$. Let $\cA'$ be a random subfamily of $\cK$ containing each element of $\cK$ independently with probability $\frac{\xi}{10}n^{1-g}$. We have \[\E[|\cA'|]=\frac{\xi}{10}n^{1-g}|\cK|\leq \frac{\xi}{10}n, \quad \E[|\cA'\cap \cK_i|]=\frac{\xi}{10}n^{1-g}\xi n^g=\frac{\xi^2}{10}n.\] By Chernoff bounds we also have $\pr {|\cA'| >\frac{\xi}{5}n}\leq e^{-\frac{\xi}{20}n}$ and $\pr{|\cA\cap \cK_i|< \frac{\xi^2}{20}n}\leq e^{-\frac{\xi^2}{80}n}$ for each $i\in [\ell]$, so by the union bound, the probability of any of these events occurring is at most $(\ell+1) e^{-\frac{\xi^2}{80}n}=o(1)$. Next, we bound the number of intersecting pairs of copies of $K_g$ in $\cA'$. We have 
            \[\E[|\{(K_1,K_2)\in \cA' \mid K_1\cap K_2\neq \emptyset\}|]\leq n^{2g-1}\left(\frac{\xi}{10}n^{1-g}\right)^2=\frac{\xi^2}{100}n.\]
            By Markov's inequality, with probability at least $1/2$ the number of such pairs is~at most $\left(\xi^2/50\right)n$. Let $\cA$ be the collection of $K_g$ copies achieved by removing one of the cliques in each intersecting pair. Notice that $|\absorber|\leq |\absorber'|< \xi n$, so $\absorber$ covers at most $\xi g n$ vertices. Moreover, with probability at least $1/2-o(1)$, the collection $\cA$ satisfies \[|\cA \cap \cK_i|\geq \frac{\xi^2}{20}n-\frac{\xi^2}{50}n\geq \frac{\xi^2}{50}n\] for each $i\in[\ell]$. Hence, there exists a family $\cA$ which satisfies the statement of the Lemma.
        \end{proof}
        \begin{claim}\label{c:many_common_neighbours_(i)}
            In both case (i) and (ii), if $\delta(G)\geq \left(1-\frac sr -\delta\right)n$, any set $S$ of at most $m$ vertices of $G$ has at least $\left(\frac{t}{r}-m\delta\right)n$ common neighbours.
        \end{claim}
            
        \begin{proof}
            It is enough to upper-bound the number of vertices that are a~non-neighbour of at least one vertex in $S$. This number is at most
        $$mn \left(\frac{s}{r}+\delta \right)= \frac{n(ms+mr\delta)}{r} = n\left(1-\frac{t-mr\delta}{r} \right).$$
        \end{proof}

        Suppose now that we are in case \ref{l:absorber(i)}. We construct the collections $\cK_i$ as follows. 
        Let $\xi_1=\frac{t}{2r}$ and $\xi=\frac{\xi_1^{m+1}}{(m+1)!}$. Note that $\left(\frac{t}{r}-m\delta\right)\geq \xi_1$. Order the pairs of vertices of $G$ arbitrarily and suppose $\{u,v\}$ is the $i$-th pair of vertices. Using Claim~\ref{c:many_common_neighbours_(i)} repeatedly, we can find vertices $w_1,\dots,w_{m-1}$ such that $w_k\in N_G(\{u,v,w_1,\dots,w_{k-1}\})$, for all $k\in[m-1]$. Notice that the set $\{w_1,\dots,w_{m-1}\}$ induces in $G$ a copy of $K_{m-1}$. By a simple counting argument, there are at least $\frac{\xi_1^{m-1}}{(m-1)!}n^{m-1}$ different choices for the set $\{w_1,\dots,w_{m-1}\}$, i.e.~we can find a collection of $\frac{\xi_1^{m-1}}{(m-1)!}n^{m-1}$ different copies of $K_{m-1}$ that are contained in $N_G(u)\cap N_G(v)$. 
        Next, extend those copies of $K_{m-1}$, again using Claim~\ref{c:many_common_neighbours_(i)} twice, to get copies of $K_{m+1}$. Each of these has at least $m-1$ vertices from $N_G(u)\cap N_G(v)$, and there are at least $\frac{\xi_1^{m+1}}{(m+1)!}n^{m+1}$ of them, by the same counting argument as before. This gives a collection $\cK_i$ of $K_{m+1}$-cliques with $|\cK_i|\geq \xi n^{m+1}$, for each $i\in [\binom n2]$. Applying Claim~\ref{c:absorber_abstract_claim} with $\ell={n\choose 2}$ and $g=m+1$, we get a collection $\cA$ which is a $(1,m,\delta_2)$-absorber for $\delta_2\leq \xi/10$.

        Next, we prove \ref{l:absorber(ii)}. Let $\xi_1=\min\left\{\delta_1,\frac{t}{2r}\right\}$ and $\xi=\frac{\xi_1^{m}}{(m+2)!}$. Notice that by Claim~\ref{c:many_common_neighbours_(i)} and the assumptions of \ref{l:absorber(ii)}, we have the following property
        \begin{itemize}
            \item[($\ast$)] for any set $S$ of $m$ vertices, $N_G(S)$ spans at least $\delta_1n$ edges in $G$.
        \end{itemize} 
        As before, using Claim~\ref{c:many_common_neighbours_(i)}, we can find $\frac{\xi_1^{m-2}}{(m-2)!}n^{m-2}$ copies of $K_{m-2}$ contained in $N_G(u)\cap N_G(v)$. Using ($\ast$), these $K_{m-2}$-copies can be extended to $K_{m}$-copies in $N_G(u)\cap N_G(v)$, and, once again using ($\ast$), these $K_m$-copies can in turn be extended to $K_{m+2}$-copies which are good for both $u$ and $v$. Hence, we get a collection $\cK_i$ of $K_{m+2}$-copies with $|\cK_i|\geq \xi n^{m+2}$. Applying Claim~\ref{c:absorber_abstract_claim} with $\ell= \binom n2$ and $g=m+2$, we get a collection $\cA$ which is a $(2,m,\delta_2)$-absorber for $\delta_2\leq \xi/10$.
    \end{proof}
    
    Next, we partition our graph with an absorber.

    \begin{lemma}\label{l:covering-with-absorber}
        Suppose $r=ms+t$ with $1<t\leq s$. Let $g = 0$ if $t<s$ and $g=1$ if $t=s$. There exists a constant $C>0$ such that for all $\gamma>0$ the following holds. There exist constants $\delta , \delta_2>0$ and $n_0\in\mathbb{N}$ such that every graph $G$ on $n$ vertices, $n\geq n_0$, with $\delta(G)\geq \left(1-\frac{s}{r}-\delta\right)n$, satisfies one of the following two conditions.
        \begin{enumerate}[label= (\roman*)]
            \item\label{l:covering-with-absorber(i)} There is a collection $\fA$ of vertex-disjoint copies of $K_{m+1+g}, Q_1^{(m+g)}, \dots , Q_{m}^{(m+g)}$ covering all but at most $C$ vertices such that $\fA$ includes a $(1+g,m,\delta_2)$-absorber, or
            \item\label{l:covering-with-absorber(ii)} There is a subset $A\subset V(G)$ with $|A|=\frac{sn}{r}$ such that $e_G(A) < \gamma n^2$. 
        \end{enumerate}
    \end{lemma}
    
    \begin{proof}
        Suppose first that $t<s$, which implies $g=0$.
        Let $C\gg r$ and $n_0^{-1}\ll \delta_{\ref{l:6.2chr}} \ll \gamma, r^{-1}$, where $\delta_{\ref{l:6.2chr}}$ is the constant obtained by applying Lemma~\ref{l:6.2chr} with $\gamma_{\ref{l:6.2chr}}=\gamma/2$. Assume that $2\delta \leq \delta_{\ref{l:6.2chr} }$. Moreover, let $\delta_2$ be sufficiently small so that $30m\delta_2 \leq \delta_{\ref{l:6.2chr} }$ and the conclusion of Lemma~\ref{l:absorber}~\ref{l:absorber(i)} holds. Due to Lemma~\ref{l:absorber}, there is a $(1,m,\delta_2)$-absorber $\fU$ for $G$. Let $G'=G-\fU$ be the graph obtained by removing vertices from the copies of $K_{m+1}$ in $\fU$ and set $n'=|V(G')|$. Since $n'\geq (1-10(m+1) \delta_2)n$ we have $\delta(G')\geq \delta(G) - 10(m+1) \delta_2 n\geq \left(1- \frac sr - \delta_{\ref{l:6.2chr} } \right)n'$. Now apply Lemma~\ref{l:6.2chr} to $G'$. We either get a collection of vertex disjoint copies of $K_{m+1},Q_1^{(m)},\dots ,Q_m^{(m)}$ covering all but at most $C$ vertices of $G'$, which implies \ref{l:covering-with-absorber(i)}, or we get an independent set of size $\left(\frac sr - \frac \gamma 2\right)n'>\left(\frac sr - \gamma\right)n$, which implies \ref{l:covering-with-absorber(ii)}, since we can extend this subset by adding arbitrary $\gamma n$ vertices.
        
        Suppose now that $t=s$, so that $g=1$.
        Similarly to the previous case, take $C\gg r$ and $n_0^{-1} \ll \delta,\delta_2 \ll \gamma, r^{-1}$. Assume that every set of size $\frac{sn}{r}=\frac{n}{m+1}$ contains at least $\gamma n^2$ edges. Applying Lemma~\ref{l:absorber}~\ref{l:absorber(ii)} to $G$ we get a  $(2,m,\delta_2)$-absorber $\fU$ for $G$. Let $G'=G-\fA$ be the graph obtained by removing vertices from the copies of $K_{m+2}$ in $\fU$ and set $n'=|V(G')|$. Since $n'\geq (1-10(m+2)\delta_2)n$ we have $\delta(G')\geq \delta(G) - 10(m+2)\delta_2 n\geq \left(1- \frac {1}{m+1} - \delta-10(m+2)\delta_2\right)n'$. Now apply Corollary~\ref{cor:6.2chr} to $G'$. We either get a collection of vertex disjoint copies of $K_{m+2},Q_1^{(m+1)},\dots ,Q_{m+1}^{(m+1)}$ covering all but at most $C$ vertices of $G'$, which implies \ref{l:covering-with-absorber(i)}, or we get an independent set of size $\left(\frac sr - \frac \gamma 2\right)n'>\left(\frac sr - \gamma\right)n$, which would imply \ref{l:covering-with-absorber(ii)}, leading to a contradiction.
    \end{proof}

    
    \section{The fully non-extremal case}\label{sec:non-extremal} 
        
    In this section, we prove Theorem~\ref{t:non-extremal-conforming}. To avoid the proof getting too long, we break it down into two arguments. The first part of the argument are standard manipulations in the reduced graph after applying Lemma~\ref{l:covering-with-absorber}. The second part, Section~\ref{sec:balancing}, is more delicate and requires new ideas. 
        
    \begin{lemma}\label{l:partition-reduced-graph}
        Let $\gamma, \zeta$ and $r \in \mathbf{N}$ be given with $0 < \zeta \ll \gamma$. Take any constants $\delta=\delta(\gamma)$, $d=d(\gamma)$, $\eps=\eps(d)$ satisfying the following hierarchy \begin{equation}\label{eq:partition-reduced-graph-constants}
            \zeta \ll \eps \ll d \ll \delta,\delta_2 \ll \gamma, r^{-1}. 
        \end{equation}
         Let $n_0=n_0(\eps)$ be sufficiently large. Let $r=ms+t$ with $m\geq 1$, $0< t\leq s$. Let $g = 0$ if $t<s$ and $g=1$ if $t=s$.
            
        Assume that $G$ is an $n$-vertex graph, $n\geq n_0$, with $\delta(G)\geq \left(1-\frac{s}{r}-\frac \delta 2\right)n$, and that every subset of $V(G)$ of size $\frac{sn}{r}$ contains at least $\gamma n^2$ edges. 
        Then $G$ has a vertex partition into sets $S_{i,j}$, $U_{k, \ell}$ and $X$, where $i \in [L_1], j \in [m+1]$, $k\in [L-L_1]$ and $\ell \in [m+1+g]$, $L_1\geq (1-10r\delta_2)L$, with the following properties. Let $R$ be the graph whose vertices are the sets $S_{i,j}$ and $U_{k,\ell}$, where two vertices $A$ and $B$ are connected if $(A, B)$ is an $(\eps, d^+)$-regular pair. Then
	\begin{enumerate}
            
            \item\label{l:part-w-r-f_(i)} For each $i \in [L_1]$ and $j < j' \in [m+1]$, the edges $S_{i,j}S_{i,j'}$ are in $R$. Similarly, for each $k\in [L-L_1]$ and $\ell < \ell' \in [m+1+g]$, the edges $U_{k,\ell}U_{k,\ell'}$ are in $R$. 
			
            \item\label{l:part-w-r-f_(ii)} 
            For each $i \in [L_1]$,
			$$|S_{i,1}|= |S_{i,2} |= \ldots = |S_{i,m}| = \frac {s}{t} \cdot |S_{i,m+1}|,$$
            $$|S_{i,1}\cup S_{i,2}\cup \dots \cup S_{i,m+1}| = \frac{n}{L}(1 \pm \eps),$$
            and $(S_{i,1},\dots, S_{i,m+1})$ is a $(\eps,d^+,d)$-super-regular tuple.

            \noindent
            For each $k \in [L-L_1]$,
			$$|U_{k,1}|= |U_{k,2} |= \ldots = |U_{k,m+1+g}|,$$
            $$|U_{k,1}\cup U_{k,2}\cup\dots\cup U_{k,m+1+g}|= \frac{n}{L}(1 \pm \eps ),$$
            and $(U_{k,1},\dots, U_{k,m+1+g})$ is a $(\eps,d^+,d)$-super-regular tuple.
			
            \item\label{l:part-w-r-f_(iii)} For $A, B \in V(R)$, there are at least $d L$ indices $k \in [L-L_1]$ such that both $A$ and $B$ have at least $m-1+g$ $R$-neighbours in $\{U_{k,1},U_{k,2},\ldots, U_{k,m+1+g} \}$.
			
            \item\label{l:part-w-r-f_(iv)} $|X|< \eps n/2$.
            
	\end{enumerate}
    \end{lemma}
    
    \begin{proof}
        Let $C=C(r)$ be the constant given by Lemma~\ref{l:covering-with-absorber}. 
        Let $d_1 = 16d$, $\eps_1=\eps^2$ and $\ell_0=\lceil 2C/\eps\rceil$. Let $n_0=n_0(\eps)$ 
        be given by Lemma~\ref{l:regularity}.
        First, apply the Regularity Lemma (Lemma~\ref{l:regularity}) with $\ell_0$,  $\eps_{\ref{l:regularity}}=\eps_1$, $d_{\ref{l:regularity}}=d_1$ and $\alpha_{\ref{l:regularity}}=1-\frac{s}{r}-\frac{\delta}2$ to $G$, to obtain an $\eps_1$-regular partition $V_0 \cup V_1 \cup \dots \cup V_{\ell'}$ with \begin{equation}\label{eq:ell'}
            \ell' \geq \ell_0 \geq \frac{2C}{\eps_1},
        \end{equation} and let $R$ be the $(\eps_1,d_1)$-reduced graph corresponding to that partition. Note that $\delta(R) \geq \left(1-\frac{s}{r}-\delta\right)\ell'$. We claim that every set of size at least $\frac sr \ell'$ in $R$ spans at least $\frac \gamma 4 \ell'^2$ edges in $R$. Indeed, it is enough to consider sets of size exactly $\frac{s}{r}\ell'$, so suppose that $I$ is such a set, i.e.~$|I| = \frac sr \ell'$, and $I$ spans in $R$ less than $\frac \gamma 4\ell'^2$ edges. Let $V_I = \bigcup_{i\in I} V_i$ and note that for each $i\in I$ we have $(1-\eps_1)\frac n{\ell'}\leq|V_i|\leq\frac{n}{\ell'}$. Then
        \begin{align*} e(G[V_I]) & \leq {|I|\choose 2}|V_1|^2(d_1+\eps_1)+\frac{\gamma}{4}\ell'^2|V_I|^2 + |I|{|V_1|\choose 2} \\ &\leq \frac{1}{2}|I||V_1|^2\left(|I|(d_1+\eps_1)+1 \right)+\frac\gamma4 n^2 \leq \frac34\gamma n^2.
        \end{align*}
        Moreover,
          \[ |V_I| \geq \frac{s}{r} (1-\eps_1)n \geq \left(\frac{s}{r} - \frac{\gamma}8\right)n.\]
          If we now take an arbitrary set of vertices $W \supset V_I$ of size $\frac{sn}{r}$, then 
          \[ e(G[W]) \leq e(G[V_I]) + \frac{\gamma}8n^2 < \gamma n^2,\]
          which contradicts the assumption on $G$.
          Note that this also implies that $R$ has no independent set of size $\left(\frac sr-\frac \gamma4\right)\ell'$.

          Recall that $g = 0$ if $t<s$ and $g=1$ if $t=s$. We can apply Lemma~\ref{l:covering-with-absorber} to $R$ to get a collection $\fA$ consisting of vertex disjoint copies of $K_{m+1+g}, Q_1^{(m+g)}, \ldots, Q_m^{(m+g)}$ and a~$(1+g, m, \delta_2)$-absorber (where $\delta_2 = \delta_2(\gamma) >0$). Let $\fU\subset \fA$ be the subcollection consisting of copies of $K_{m+1+g}$ in the absorber, and let $\fS = \fA \setminus \fU$ be the remaining part. Note that by the definition of the absorber, $|V(\fU)|\leq 10(m+1+g)\delta_2|V(R)|=10(m+1+g)\delta_2\ell'$. Moreover, let $Y$ be the set of vertices of $R$ not covered by $\fA$ and note that $|Y|\leq C$. The next step is to subdivide the sets in $V(\fS)$. To this end, we use the $\left\{\frac 1b\ltimes T\right\}$-factor given by Corollary~\ref{cor:T_packing}.

          The $\left\{\frac 1b\ltimes T\right\}$-factor of $\fS$ is given by a collection of vertex sets
          \[ \left(\{\sigma_1^i,\dots,\sigma_{m+1}^i\}\right)_{i\in[L_1]}\] 
          for a specific $L_1$ to be determined in a moment, and a function
          \[ f: \bigcup_{j\in[m+1],\,i\in[L_1]} \{\sigma_{j}^{i}\} \to V({\fS})\]
          such that $f(\sigma_j^i)f(\sigma_{j'}^i)$ is an edge of $R$ for each $i\in[L_1]$ and $j,j'\in[m+1]$, $j\neq j'$. 
          Moreover, since the total weight of $\frac 1b\ltimes T$ is $\frac 1b$, and a $\left\{\frac 1b\ltimes T\right\}$-factor has residue $0$, we have $L_1 = b|V(\fS)|$. We now define a family of sets $S'_{i,j}\subseteq V(G)$ corresponding to vertices $\sigma^i_j$, for $i\in[L_1]$ and $j\in[m+1]$. For a set $S\in V(\fS)$, each vertex $x\in S$ randomly chooses one of the sets $S'_{i,j}$ if and only if $f(\sigma^i_j)= S$. Moreover, if this is the case, $x$ chooses $S'_{i,j}$ with probability 
          \[\frac 1b\cdot w(\sigma^i_j) = \begin{cases}
              \frac{s}{br}, & j\in[m] \\
              \frac{t}{br}, & j = m+1.
          \end{cases}\]
        Since $f$ is a $\left\{\frac 1b\ltimes T\right\}$-factor, this selection gives us a probability measure for each vertex $x \in \bigcup_{S\in V(\fS)} S$. Hence, using Chernoff bounds, we get that w.h.p.
        \[ |S'_{i,j}| = \begin{cases}
              \frac{s}{br}|V_1|(1\pm\eps_1), & j\in[m] \\
              \frac{t}{br}|V_1|(1\pm\eps), & j = m+1,
          \end{cases}\]
          and
          \[ |S'_{i,1}\cup S'_{i,2} \cup \dots \cup S'_{i,m+1}| = \frac 1b|V_1|(1\pm\eps_1).\]
         Next, we split the sets in $V(\fU)$ into sets $U'_{k,\ell}$, for $k\in[L_2]$ and $\ell\in[m+1+g]$, using a~similar approach. This time each set $U\in V(\fU)$ is split randomly into $b(m+1+g)$ sets so that we get $L_2 = b(m+1+g)|\fU| = b|V(\fU)|\leq 10b(m+1+g)\delta_2\ell'$ and
         \[ |U'_{k,\ell}| = \frac{1}{b(m+1+g)}|V_1|(1\pm\eps_1),\]
         and
         \[ |U'_{k,1}\cup U'_{k,2} \cup \dots \cup U'_{k,m+1+g}| = \frac 1b|V_1|(1\pm\eps_1).\]
         Note that 
         \[ L := L_1+L_2 = b|V(\fS)| + b|V(\fU)| = b(\ell' - C),\]
         hence
         \[ L_1 \geq \left(1-10r\delta_2\right)L\]
         and
         \[ \frac 1b|V_1|(1\pm\eps_1) = \frac{n}{L}(1\pm 3\eps_1).\]
         
        Next, by Lemma~\ref{l:slicing}, all pairs stemming from the edges of $R$ are $\left(\eps_1/\beta, \left(d_1/4\right)^+\right)$-regular with $\beta = \min\left\{\frac{t}{2br}, \frac{1}{2b(m+1+g)}\right\}\leq\frac12$. This includes all pairs $S'_{i,j}S'_{i,j'}$ and $U'_{k,\ell}U'_{k,\ell'}$, but also some pairs $S'_{i,j}U'_{k,\ell}$. 

        Now to get super-regularity consider first a tuple $(S'_{i,1},\dots,S'_{i,m+1})$. Using Lemma~\ref{l:super-reg-k-tuple}, we can discard at most $(m+1)\frac{\eps_1}{\beta}$ vertices from each part and pass to subsets $S''_{i,j}\subset S'_{i,j}$. In a similar way we can discard at most $(m+1+g)\frac{\eps_1}{\beta}$ vertices from each part $U'_{k,\ell}$ and pass to subsets $U''_{k,\ell}\subseteq U'_{k,\ell}$. After this procedure the pairs in corresponding tuples are $\left(2\eps_1/\beta, \left(d_1/8\right)^+, d_1/8\right)$-super-regular.
          
        Finally, we again discard at most $\eps_1^{3/4}$-fraction of vertices from each part to get the correct sizes, i.e.~we pass to sets $S_{i,j}\subseteq S''_{i,j}$ and $U_{k,\ell}\subseteq U''_{k,\ell}$ satisfying \[|S_{i,1}|= \ldots = |S_{i,m}| = \frac {s}{t} \cdot |S_{i,m+1}|\] 
        for each $i\in[L_1]$ and 
        \[|U_{k,1}|= \ldots = |U_{k,m+1+g}|\] 
        for each $k\in[L_2]$. After this step all super-regular pairs remain super-regular, but with different parameters, namely they are $\left(\sqrt{\eps_1}, \left(d_1/16\right)^+,d_1/16\right)=\left(\eps,d^+,d\right)$-super-regular. Similarly, all other regular pairs remain regular with parameters $\left(\sqrt{\eps_1}, \left(d_1/16\right)^+\right)=\left(\eps,d^+\right)$.
         
         Let $\cS = \bigcup_{i,j}\{S_{i,j}\}$ and $\cU = \bigcup_{k,\ell}\{U_{k,\ell}\}$. We define a new graph $R'$ on vertex set $\cS\cup\cU$ where edges correspond to $\left(\eps, d^+\right)$-regular pairs. Note that by our construction, $R'$ satisfies properties \ref{l:part-w-r-f_(i)} and \ref{l:part-w-r-f_(ii)} from the lemma statement. Intuitively, one can think of $R'$ as a blow-up of $R$.

         We now show property \ref{l:part-w-r-f_(iii)}. Take any $A, B\in \cS\cup\cU$. Let $V_a$ and $V_b$ be the original parts of the initial partition such that $A\subset V_a$ and $B \subset V_b$. Since $\fU$ is the $(1+g,m,\delta_2)$-absorber in $R$, there are at least $\delta_2^2\ell'$ copies $\tilde{K}$ of $K_{m+1+g}$ in $\fU$ which are good for both $V_a$ and $V_b$, that is both $V_a$ and $V_b$ have at least $m-1+g$ $R$-neighbours in $\tilde{K}$. Let $\tilde{U}_\ell$, $\ell\in[m+1+g]$, be the vertices of $\tilde{K}$ and $\tilde{U}_{k,\ell}$, $k\in\left[b(m+1+g) \right]$, $\ell\in[m+1+g]$ be the sets in $\cU$ with $\tilde{U}_{k,\ell} \subset \tilde{U}_\ell$. Then if $V_a\tilde{U}_
        \ell$ is an edge in $R$, that is $(V_a,\tilde{U}_
        \ell)$ is an $(\eps_1, d_1)$-regular pair, then $(A, \tilde{U}_{k,\ell})$ is a $\left(\eps, d^+\right)$-regular pair, which means that $A\tilde{U}_{k,\ell}$ is an edge in $R'$. The same is true for $B$. Hence, taking $d<\frac 12 \delta_2^2,$ there are at least $\delta_2^2\ell'b(m+1+g) \geq dL$ indices $k\in[L_2]$, such that both $A$ and $B$ have at least $m-1+g$ $R'$-neighbours in $\{U_{k,1}, U_{k,2},\dots,U_{k,m+1+g}\}$. 

        Finally, to prove \ref{l:part-w-r-f_(iv)}, let $X$ be the set of all vertices of $G$ that do not belong to any of the sets $S_{i,j}$ nor $U_{k,\ell}$. By our construction, $X$ consists of:
         \begin{itemize}
             \item vertices in $V_0$ (the initial left-over set from the Regularity Lemma),
             \item vertices discarded to get super-regularity,
             \item $\bigcup_{i\in Y}V_i$ (vertices not covered by the collection $\fA$),
             \item vertices discarded to get correct sizes of the parts for condition \ref{l:part-w-r-f_(ii)}.
         \end{itemize}
         Hence, using \eqref{eq:ell'},
         \[ |X|< \eps_1 n + \frac{(m+1+g)\eps_1}{\beta}n + C\frac{n}{\ell'}(1\pm\eps_1) + \eps_1^{3/4}n \leq \sqrt{\eps_1} n/2 = \eps n/2.\]
        \end{proof}

    \subsection{Size adjustment lemmas}

    The following lemmas are all elementary and will be used to adjust the sizes of the partition sets by removing a constant or small fraction of vertices to obtain appropriate size ratios. 
    
    \begin{lemma}\label{l:equalize-(m+1)-tuples-to-s/t}
        Let $S_1,\dots , S_m , T$ be disjoint sets of vertices and let $M\coloneq |S_1\cup\dots\cup S_m\cup T|$. Let $0<\epsilon\ll r^{-1}$. Suppose that $M\gg r$, $M\equiv 0 \pmod{r}$ and $\left(\frac{s}{r}-\epsilon\right) M\leq|S_i|\leq \frac{s}{r}M$ for all $i\in [m]$.
        Consider the following move:
        \begin{itemize}
            \item $P_k$ move: for an index $k\in [m]$ remove $s-1$ vertices from $S_k$, $s$ vertices from each $S_i$, $i\neq k$, and $t+1$ vertices from $T$.
        \end{itemize}   
        Using the above type of move, we can get new sets $S_1'\subseteq S_1,\dots,S_m'\subseteq S_m,T'\subseteq T$ such that $|S_1|=\dots =|S_m|=\frac{s}{t}|T|$ and $|S_1'\cup\dots\cup S_m'\cup T'|\geq (1-mr\epsilon)M.$
    \end{lemma}

    \begin{proof}
        Suppose we can obtain the required sets $S_i'$, $i\in[m]$, and $T'$ by performing $x_k$ moves $P_k$ for $k\in[m]$. This leads to a system of linear equations
        \begin{align}\label{eq:system_x_k}
            |S_1'|&=\frac{s}{t}|T'| \nonumber\\
            &\,\,\,\,\vdots \\
            |S_m'|&=\frac{s}{t}|T'|, \nonumber
        \end{align}
        where $|S_k'|=|S_k|-s\sum_{i=1}^m x_i+x_k$ for $k\in[m]$, and $|T'|=|T|-(t+1)\sum_{i=1}^m x_i$. So it is enough to check if \eqref{eq:system_x_k} has a solution in non-negative integers. Indeed, there is such a~solution given by
        \[x_k=\frac{s}{r}M-|S_k|, \quad \text{ for }k\in[m].\] Since $\frac{s}{r}M-|S_k|\leq \epsilon M$, and in a single move we remove exactly $r$ vertices, the total number of removed vertices is at most $mr\epsilon M$ vertices.
    \end{proof}

     \begin{lemma}\label{l:absorber-divisibility-by-r}
         Let $U_1,\dots  , U_{m+1}$ be disjoint sets of vertices and let $M\coloneq |U_1\cup\dots\cup  U_{m+1}|$. Let $0<\epsilon \ll 1 $. Suppose that $M\gg r$,  $M \equiv 0 \pmod{r}$ and $|U_i|=\frac{M}{m+1}(1\pm\epsilon)$ for all $i\in [m]$.
        Consider the following two types of moves:
        \begin{itemize}
            \item $P_j$ move: for an index $j\in [m]$ remove $t$ vertices from $U_j$ and $s$ vertices from each $U_i$, $i\in[m+1]\setminus\{j\}$.
            \item $P_{j,k}$ move: for indices $j,k \in [m]$ remove $t+1$ vertices from $U_j$, $s-1$ vertices from $U_k$ and $s$ vertices from each $U_i$, $i\in[m+1]\setminus\{j,k\}$.
            \end{itemize} 
        Using the above moves, we can get new sets $U_1'\subseteq U_1,\dots,U_{m+1}'\subseteq U_{m+1}$ such that each $|U_i'|$ is divisible by $r$ and $|U_1'\cup\dots\cup U_{m+1}'|\geq M -(m+1)^2r^2$.
     \end{lemma}

    \begin{proof}
        We will perform the following algorithm using the allowed moves. Suppose that we have reached an $(m+1)$-tuple $W_1,\dots,W_{m+1}$. Associate an $(m+1$)-tuple $(y_1,\dots,y_{m+1})$ of remainders in $\{-(r-1),\dots,r-1\}$ to $W_1,\dots, W_{m+1}$ so that $|W_i|\equiv y_i \pmod r$ and $\sum y_i=0$. Furthermore, assume that this choice is minimal in terms of $\sum |y_i|$. If $\sum|y_i|=0$ then we stop the algorithm. Otherwise, without loss of generality, assume that $y_1\geq \dots \geq y_{m+1}$ and note that $y_1>0>y_{m+1}$. Apply the moves $P_{1,m+1}$, $P_2,\dots ,P_m,P_{m+1}$. With this sequence of moves, we have removed $r+1$ vertices from $W_1$, $r-1$ vertices from $W_{m+1}$, and $r$ vertices from each $W_i$, $i=2,\dots ,m$. Thus, we have reduced $|y_1|$ and $|y_{m+1}|$ each by $1$, and the sum $\sum |y_i|$ by $2$. By iterating this process, we can reach $\sum |y_i|=0$ in a finite number of steps and finish the algorithm.

        Note that at the end of the algorithm we reach an $(m+1)$-tuple with each set of size divisible by $r$.
        Moreover, since in each iteration of the algorithm we remove $(m+1)r$ vertices, and the number of iterations is at most $\frac12(r-1)(m+1)<r(m+1)$, we have removed in total at most $(m+1)^2r^2$ vertices. 
    \end{proof} 

    \begin{lemma}
    \label{l:sing-absorber-divisibility-by-s}
        Suppose that $r=(m+1)s$. Let $U_1,\dots, U_{m+2}$ be disjoint sets of vertices and let $M:=|U_1\cup\dots\cup U_{m+2}|$. Let $0<\epsilon \ll 1$. Suppose that $M\gg r$, $M \equiv 0 \pmod{r}$ and $|U_i|= \frac{M}{m+2}(1\pm \epsilon)$ for all $i\in [m+2]$. Consider the following type of move:
        \begin{itemize}
            \item $P_{j,k}$ move: for indices $j,k\in [m+2]$ remove $1$ vertex from $U_j$, $s-1$ vertices from $U_k$ and $s$ vertices from each $U_i$, $i\in[m+2]\setminus\{j,k\}$.
        \end{itemize}
        Using the above moves, we can get new sets $U_1'\subseteq U_1,\dots , U_{m+2}'\subseteq U_{m+2}$ such that each $|U_i'|$ is divisible by $s$ and $|U_1'\cup \dots \cup U_{m+2}'|\geq M-(m+2)s^2$.
    \end{lemma}

        The proof is essentially the same as the proof of Lemma~\ref{l:absorber-divisibility-by-r} so we omit it.

    \begin{lemma}\label{l:sing-absorber-partitioning}
        Suppose that $r=(m+1)s$. Let $U_1,\dots, U_{m+2}$ be disjoint sets of vertices and $M:=|U_1\cup\dots\cup U_{m+2}|$. Let $0<\epsilon \ll 1$. Suppose that $M\gg r$, $|U_i|\equiv 0 \pmod{s}$, $M\equiv 0 \pmod{r}$ and $|U_i|= \frac{M}{m+2}(1\pm \epsilon)$ for all $i\in [m+2]$. We can partition each set $U_i$ into $m+1$ parts $P_{i,1},\dots, P_{i,i-1}, P_{i,i+1},\dots, P_{i,m+2}$ so that the size of each part is divisible by $s$, and for each $j\in [m+2]$, the sizes $|P_{i,j}|$ are equal for all $i\neq j$.
    \end{lemma}

    \begin{proof}
          Set 
          \[x_i=\frac{1}{m+1}\sum_{j=1}^{m+2}|U_j|-|U_i| = \frac{M}{m+1}-|U_i|\] 
          and note that each $x_i$ is divisible by $s$. Now simply partition each $U_i$ into parts $P_{i,j}$ for $j\neq i$ so that $|P_{i,j}|=x_j$. Note that any partition satisfying the size requirements works.
    \end{proof}

    \subsection{From regular partition to clique factor}
        \label{sec:balancing}
        
    We are now ready to prove Theorem~\ref{t:non-extremal-conforming}.

    \begin{proof}[Proof of Theorem~\ref{t:non-extremal-conforming}]
    Consider the graph $\tilde G = G\setminus Z$ and note that since $|Z|=\zeta n$ and $\zeta \ll \delta \ll \gamma$, we have $\delta(\tilde G)\geq\left(1-\frac{s}{r} - 2\delta\right)|V(\tilde G)|$ and every subset of size $\frac{s}{r}|V(\tilde G)|$ induces in $\tilde G$ at least $\frac{\gamma}2 |V(\tilde G)|^2$ edges.
    Hence, we can apply to $\tilde G$ Lemma~\ref{l:partition-reduced-graph} with constants $\delta_{\ref{l:partition-reduced-graph}}=2\delta$ and $\gamma_{\ref{l:partition-reduced-graph}}=\frac{\gamma}2$. We thus obtain sets
    \begin{itemize}
        \item $S_{i,j}$, for $i\in[L_1]$, $j\in[m+1]$,
        \item $U_{k,\ell}$, for $k\in[L-L_1]$, $\ell\in[m+1+g]$,
        \item $X$,
    \end{itemize} 
    and a reduced graph $R$ corresponding to $(\eps,d^+)$-regular pairs. Let $\tilde X = X \cup Z$ and note that $|\tilde X|\leq \eps n$. Instead of sampling $\Gnp$ with $p = Cp_s(n)$ at once, we will expose three independent random graphs $G_i \sim G(n, p/4)$ for $i \in [3]$. The union $G_1 \cup G_2 \cup G_3$ can be viewed as a subgraph of $\Gnp$.
    
    The proof consists of four main steps.
        
    \smallskip
 
    \textit{Step 1. Redistributing $\tilde X$.} We redistribute the elements of $\tilde X$ among the sets $S_{i,j}$ and get new parts $S_{i,j}'\supseteq S_{i,j}$ which are of similar size as the original ones. Here we ensure that we maintain super-regularity between all sets in $\{S_{i,j}'\}_{j\in[m+1]}$, for $i\in[L_1]$.

    \smallskip

    \textit{Step 2. Readjusting the sizes of the sets $S_{i,j}'$.} Using appropriate absorbers (in sets $U_{k,\ell}$) and random edges of $G_1$, we take out some copies of $K_r$, each using at most one vertex from $\tilde X$, to return to exact appropriate size ratios. This yields new parts $S_{i,j}^\dagger \subseteq S_{i,j}'$ and $U_{k,\ell}^\dagger \subseteq U_{k,\ell}'$.

    \smallskip

    \textit{Step 3. Ensuring the divisibility of $|U_{k,1}^\dagger\cup\dots\cup U_{k,m+1+g}^\dagger|$.} We create a move $U-A-U'$, where $U$, $U'$ are parts of the regular partition and $A$ is an absorbing $(m+1+g)$-tuple. This move takes out a copy of $K_r$ which uses random edges in $G_2$ and in turn changes $|U|$ and $|U'|$ by $-1 $ and $+1\pmod{r}$, respectively. We can iterate it a constant number of times to ensure that the new parts $U_{k,\ell}^* \subseteq U_{k,\ell}^\dagger$ satisfy $|U_{k,1}^*\cup\dots\cup U_{k,m+1+g}^*| \equiv 0 \pmod{r}$, for all $k\in[L_2]$ $(L_2:=L-L_1)$.

    \smallskip

    \textit{Step 4. Covering the remainder using $G_3$.} In the last step we cover the remaining super-regular tuples of appropriate sizes using random edges in $G_3$.

    \smallskip

    We now explain each step in detail.

    \smallskip

    \textbf{Step 1.~Redistributing $\tilde X$.} 
    
    \smallskip
    
    Call an index $i \in [L_1]$ \textit{good for} $x$ if there is at most one index $j\in[m+1]$ such that $x$ has fewer than $d|S_{ij}|$ neighbours in $S_{ij}$. We claim that for each $x\in \tilde X$, there are at least ${dL}$ good indices. Indeed, assume that the opposite is true. Any index $i\in[L_1]$ that is not good for $x$ must yield at least two indices $j\in[m+1]$ for which $x$ has more than 
    \[(1-d)|S_{i,j}| = \begin{cases}
        (1-d)\frac sr \cdot \frac nL (1\pm \eps), & j\in[m] \\
        (1-d)\frac tr \cdot \frac nL (1\pm \eps), & j=m+1
    \end{cases}\] 
    non-neighbours in $S_{i,j}$. Then, since $L_1\geq (1-10r\delta_2) L$ and $\delta_2 \ll d$, the number of non-neighbours of $x$ is at least
		$$(1-10r\delta_2-d)L\cdot(1-d)\frac{s+t}{r}\cdot\frac{n}{L}(1-\eps)>\frac{ns}{r}.$$ 
	This contradicts the minimum-degree assumption on $x$.

    In order to assign good indices to vertices in $\tilde X$, we will use the following standard argument which will also be used later in the proof.
    \begin{claim}\label{c:even-distribute-X}
        There is an assignment $f: \tilde X \to L_1$ such that for each $x\in \tilde X$, $f(x)$ is good for $x$, and for each $i \in [L_1]$, $|f^{-1}(i)| \leq \frac{\eps n}{dL}$.
    \end{claim}
    
    \begin{proof}
        Consider a random assignment which for each vertex $x\in \tilde X$ chooses a good index $i\in[L_1]$ uniformly at random and sets $i = f(x)$. Now fix an index $i$. Since $\pr{f(x) = i} \leq (dL)^{-1}$, the random variable $|f^{-1}(i)|$ is stochastically dominated by $\Bin(\eps n/2, (dL)^{-1})$, and hence by Chernoff
        $$\pr{|f^{-1}(i)| \geq \frac{\eps}{d}\frac{n}{L}} = e^{-\Omega(n)}.$$ 
        Taking the union bound over all $i\in[L_1]$ implies that $f$ satisfies the conclusion of the claim with positive probability.
    \end{proof}

    Now we will redistribute the vertices of $\tilde X$, creating a new partition that satisfies properties similar to \ref{l:part-w-r-f_(i)}--\ref{l:part-w-r-f_(iv)}. For this purpose, we will use Lemma~\ref{l:slicing-adding}. 

    \begin{claim}
        Let $\eps_1= \sqrt \eps / d$. There is a vertex partition of $G$ into sets $S'_{i,j}$ and $U'_{k,l}$ (with the same indexing as before) with the following properties. Let $R'$ be the graph whose vertices are sets $S'_{i,j}$, $U'_{k,l}$ and two vertices $A$ and $B$ are connected if $(A,B)$ is an $\left(\eps_1,\left(d/2\right)^+\right)$-regular pair.
        \begin{enumerate}[label=(\roman*')]
            \item\label{claim_(i)} For each $i \in [L_1]$ and $j < j' \in [m+1]$, the edges $S'_{i,j}S'_{i,j'}$ are in $R'$. Similarly, for each $k\in [L_2]$ and $\ell < \ell' \in [m+1+g]$, the edges $U'_{k,\ell}U'_{k,\ell'}$ are in $R'$.
            
            \item\label{claim_(ii)} For each $i \in [L_1]$, $j\in[m]$ we have 
            $$|S'_{i,j}| = \frac{s}{r}\cdot\frac{n}{L}(1\pm \eps_1),\quad |S'_{i,m+1}| = \frac{t}{r}\cdot\frac{n}{L}(1\pm \eps_1)$$
            and $(S'_{i,1},\dots, S'_{i,m+1})$ is an $\left(\eps_1,\left(d/2\right)^+,d/2\right)$-super-regular tuple.
            
            \noindent
            For each $k \in [L_2]$, $\ell\in[m+1+g]$ we have 
            $$|U'_{k,\ell}|=\frac {1}{m+1+g}\cdot\frac{n}{L}(1 \pm \eps_1)$$
            and $(U'_{k,1},\dots, U'_{k,m+1+g})$ is an $\left(\eps_1,\left(d/2\right)^+,d/2\right)$-super-regular tuple.
            
            \item\label{claim_(iii)}  For $A, B \in V(R')$, there are at least $d L$ indices $k \in [L_2]$ such that both $A$ and $B$ have at least $m-1+g$ $R'$-neighbours in $\{U'_{k,1}, U'_{k,2}, \ldots, U'_{k,m+1+g}\}$.
            
            \item\label{claim_(iv)} Denote $s_{ij}:=|S_{i,j}'|-|S_{i,j}|$. Then, $$\sum_{i,j}s_{ij}\leq \eps n.$$ 
            
            \item\label{claim_(v)} For all $i\in[L_1]$, $j\in [m+1]$, 
            $$|\tilde X\cap S_{i,j}'|\leq \frac{\eps}{d}|S_{i,j}'|.$$
        \end{enumerate}
    \end{claim}
             
    \begin{proof}
        Recall that 
        \[ |S_{i,j}|=\frac{s}{r}\cdot\frac{n}{L}(1\pm\eps), j\in[m] \quad \text{ and } \quad |S_{i,m+1}|=\frac{t}{r}\cdot\frac{n}{L}(1\pm\eps). \]
        Similarly,
        \[ |U_{k,\ell}|=\frac1{m+1+g}\cdot\frac{n}{L}(1\pm\eps).\]
        Using Claim~\ref{c:even-distribute-X}, we can redistribute the vertices of $\tilde X$ into $S_{i,j}$ to obtain new sets $S_{i,j}'$ as follows. By the definition of good indices, any $x\in \tilde X$ can be added to some part $S_{f(x),j}$, so that it has at least a $d$ proportion of neighbours in each part $S_{f(x),j'}$ for $j' \neq j$. Now taking $\eps_1 = \sqrt{\eps}/d$, Claim~\ref{c:even-distribute-X} asserts that no more than 
        \[\frac{\eps n}{dL} < \eps_1\sqrt{\eps}|S_{i,j}|< \frac{\eps_1}2|S_{i,j}|\] vertices were added to each $S_{i,j}$. Moreover, set $U'_{k,\ell}=U_{k,\ell}$ for all $k\in[L-L_1]$, $\ell\in[m+1+g]$.
        Properties \ref{claim_(ii)} and \ref{claim_(v)} follow immediately from the above observations, where additionally we use the fact that vertices in $\tilde X$ are redistributed in such a way, that the super-regularity is preserved between suitable pairs with parameter $d/2$. 
        
        Next, property \ref{claim_(iii)} follows from property \ref{l:part-w-r-f_(iii)} using Lemma~\ref{l:slicing-adding}, i.e.~adding an $\eps/d$ proportion of vertices to any of the sets $S_{i,j}$ turns $(\sqrt{\eps},d)$-regular pairs into $(\sqrt\eps/d, d')$-regular pairs with $d'\geq d-3\sqrt \eps$. Similarly property \ref{claim_(i)} follows from property \ref{l:part-w-r-f_(i)}. Finally, the sum in \ref{claim_(iv)} is just $|\tilde X|\leq \eps n $.
    \end{proof} 

    \smallskip

    \textbf{Step 2.~Readjusting the sizes of the sets $S_{i,j}'$.}  
    
    \smallskip
    
    In this step, we will remove a small number of vertices from the sets $S_{i,j}'$ and $U'_{k,\ell}$, forming disjoint copies of $K_r$, and obtain a new collection of sets $S_{i,j}^\dagger\subseteq S_{i,j}'$ and $U_{k,\ell}^\dagger\subseteq U_{k,\ell}'$ satisfying
    \begin{equation}\label{eq:correct_proprtion_in_S_dagger}
        |S_{i,1}^\dagger|= |S_{i,2}^\dagger |= \dots = |S_{i,m}^\dagger| = \frac {s}{t} \cdot |S_{i,m+1}^\dagger|
    \end{equation}
    for all $i\in[L_1]$. To this end we define the following procedure. For $i\in [L_1]$, $j\in[m+1]$ and $k\in [L_2]$, in a \textit{basic $(i,j;k)$-move} we remove a copy of $K_r$ in $G \cup G_1$ which contains one vertex from $S_{i,j}'$ and $r-1$ vertices from $U_{k,1}'\cup \dots \cup U_{k,m+1+g}'$. We will slightly abuse the notation by referring to this move as a basic $(i,j;k)$-move even after some vertices have been already removed from the corresponding vertex sets $S_{i,j}'$ or $ U_{k,\ell}'$. 

    Recall that $G_1\sim G(n, Cp_s)$. Before we proceed, we define the following property of the graph $G_1$, which, as we show below, holds w.h.p. Let $d_1$ be a constant with $\eps \ll d_1 \ll d$.
    \begin{description}
	\item[(P1)] For $i=0,1,\ldots,m+3$, let $A_i\in V(R')$ and let $B_i \subset A_i$ with $|B_i| \geq |A_i|/2$. Assume that all pairs $(A_i, A_j)$ are $(\eps^{1/10}, d')$-regular for some $d' \geq d_1$, with (possibly) the exception of $(A_0, A_{1})$ and $(A_0, A_{2})$. Then $G \cup G_1$ contains a clique with exactly one vertex in $B_0$, $s-1$ vertices in each $B_{1}$ and $B_{2}$, and $s$ vertices in $B_i$, $i = 3, \dots, m+3$.
    \end{description}
			
    \begin{claim}\label{c:p1p2}
	$G_1$ satisfies (P1) with probability $1-e^{-\omega(n)}$.
    \end{claim}

    \begin{proof}
        Let $H$ be a graph with vertex set $U_0\cup U_1 \cup\dots\cup U_{m+3}$, where $|U_0|=1$, $|U_1|=|U_2|=s-1$, and $|U_3|=\dots=|U_{m+3}|=s$, where all pairs apart from $(U_0,U_1)$ and $(U_0, U_2)$ span a complete bipartite graph.
        Let $\mathcal{H}$ be the family of copies of $H$ in $G$ with the part corresponding to $U_i$ contained in $B_i$, $i=0,\dots,m+3$. Using the Counting Lemma we have $|\mathcal{H}|\geq \nu n^{|V(H)|}$, for a specific constant $\nu$. 
        We claim that after revealing the edges of $G_1$, w.h.p.~at least one copy of $H\in \mathcal{H}$ becomes the desired clique. Indeed, note that the complement $\bar{H}$ of $H$ consists of $m+3$ copies of $K_s$, with exactly two copies sharing exactly one vertex and the remaining copies being vertex disjoint. Then by Lemma~\ref{l:many_ks_and_one_ks-ks}, the failure probability that no element in $\mathcal{H}$ spans a~complete graph in $G\cup G_1$ is at most $e^{-\Theta(n^2p)}=e^{-\omega(n)}$. Taking the union bound over all $2^{O(n)}$ choices of sets $B_i$ and $O(1)$ choices of $A_i\in V(R')$, we get that (P1) holds with probability $1-e^{-\omega(n)}$.
    \end{proof}

    Henceforth we assume that $G_1$ indeed satisfies (P1). Notice that (P1) implies the following.
    
    \begin{remark}\label{rem:(P1)}
    If $g=0$ 
    (the regular case $r = ms+t$ with $t < s$), 
    then the graph $G \cup G_1$ contains:
    \begin{itemize}
        \item a~$K_{r}$-copy in $A_0 \cup A_1 \cup \dots \cup A_{m+1}$ with exactly one vertex in $A_0$,
        \item a~$K_{r}$-copy in $A_1 \cup A_2 \cup \dots \cup A_{m+1}$ with exactly $s-1$ vertices in $A_1$ and $t+1\leq s$ vertices in $A_{m+1}$,
        \item a~$K_{r+1}$-copy in $A_3 \cup A_4 \cup \dots \cup A_{m+3}$ with exactly $t+1$ vertices in $A_{m+3}$ and $s$ vertices in the remaining sets.
    \end{itemize} 
        
    If $g=1$ 
    (the singular case $r = ms+t$ with $t = s$), 
    then the graph $G \cup G_1$ contains:
    \begin{itemize}
        \item a $K_{r}$-copy in $A_0 \cup A_2 \cup A_3 \cup \dots \cup  A_{m+2}$ with exactly one vertex in $A_0$,
        \item a~$K_{r+1}$-copy in $A_0\cup A_3 \cup A_4 \cup \dots \cup A_{m+3}$ with exactly $1$ vertex in $A_0$ and $s$ vertices in the remaining sets.
    \end{itemize} 
    \end{remark}
    
    We will now use this property to perform a series of $(i,j;k)$-moves which will lead us to a~partition satisfying \eqref{eq:correct_proprtion_in_S_dagger}. 
        
    \begin{claim}
        It is possible to perform an $(i,j;k)$-move as long as $S_{i,j}'$ forms a $(2\eps_1, d')$-regular pair, for some $d'\geq d/4$, with at least $m-1+g$ sets $U_{k,\ell}'$.
    \end{claim} 
        
    \begin{proof}
        If $g=0$ (regular case), w.l.o.g.~assume that $S_{i,j}'$ is $(2\eps_1,d')$-regular with the sets $U_{k, 3}',\dots,U_{k,m+1}'$. We apply (P1) with $A_0=S_{i,j}', A_1=U_{k,1}', \ldots, A_{m+1}=U_{k,m+1}'$ to remove a copy of $K_r$ with $t\leq s-1$ vertices in $A_1$ and $s-1$ vertices in $A_2$. If $g=1$, we similarly apply (P1) with $A_0=S_{i,j}', A_2=U_{k,1}', A_3=U_{k,2}', \dots, A_{m+2}=U_{k,m+1}'$.
    \end{proof}

    In order to be able to perform a series of $(i,j;k)$-moves, we need to show that while removing a small number of vertices in the copies of $K_r$, regularity is maintained. Note that by removing at most an $\eps_1$ proportion of vertices from any partition set, using Lemma~\ref{l:slicing}, we can ensure that pairs which were $(\eps_1,(d/2)^+)$-regular are $(2\eps_1,(d/4)^+)$-regular in the new partition, and, analogously, pairs that were $(\eps_1,(d/2)^+,d/2)$-super-regular become $(2\eps_1,(d/4)^+,d/4)$-super-regular. Hence, we need to perform the $(i,j;k)$-moves in such a way that each partition set loses at most an $\eps_1$ fraction of its vertices.

    Recall that $s_{ij} = |S_{i,j}'|-|S_{i,j}| \leq \frac{\eps}{d}|S_{i,j}'| \leq \eps_1|S_{i,j}|/2$. For each part $S_{i,j}'$, we will remove $s_{ij}$ vertices by performing a sequence of basic $(i,j;k)$-moves. However, we don't want to use any of the vertices $k\in[L_2]$ too often, so that the total number of vertices removed from any set $U_{k,\ell}'$ is also at most an $\eps_1$ fraction. Recall that by \ref{claim_(iii)}, for each $S_{i,j}'$, there are at least $dL$ indices $k$ for which an $(i,j;k)$-move can be performed. Hence, by the same argument as in Claim~\ref{c:even-distribute-X}, and using \ref{claim_(iv)}, we can assign $s_{ij}$ indices $k\in[L_2]$, possibly with repetitions, to each $S_{i,j}'$ so that none of the indices is assigned more than $\frac{2\eps n}{dL}$ times. Now for each assigned pair $S_{i,j}'$ and $k$, we apply a basic $(i,j;k)$ move. This removes exactly $s_{ij}\leq \eps_1|S_{i,j}|/2$ vertices from each set $S_{i,j}'$ and at most an $4r\eps/d < \eps_1$ proportion of vertices from any given $U_{k,\ell}'$.
    
    Denote the resulting sets by $S_{i,j}^\dagger$ and $U_{i,j}^\dagger$. Then \eqref{eq:correct_proprtion_in_S_dagger} follows by \ref{l:part-w-r-f_(ii)} and the fact that for each $i\in[L_1]$ and $j\in[m+1]$ we have $|S_{i,j}^\dagger|=|S_{i,j}|$. Note also that each of the removed copies of $K_r$ contains at most one vertex from $\tilde X$.
    
    Denote $\Ub_k^\dagger := U_{k,1}^\dagger \cup \dots \cup U_{k,m+1+g}^\dagger$, for $k\in[L_2]$. Let $R^\dagger$ be the graph with vertex set $U_{k,\ell}^\dagger$ and edges corresponding to $(2\eps_1,(d/4)^+)$-regular pairs. Then by our construction and \ref{claim_(i)}, \ref{claim_(ii)} and \ref{claim_(iii)}, we have
    \begin{enumerate}[label=(\roman*$^\dagger$)]
        \item\label{(i)dagger} for each $k\in [L_2]$ and $\ell < \ell' \in [m+1+g]$, the edges $U^\dagger_{k,\ell}U^\dagger_{k,\ell'}$ are in $R^\dagger$;
        \item\label{(ii)dagger} for each $k \in [L_2]$, $\ell\in[m+1+g]$ we have 
            $$|U^\dagger_{k,\ell}|=\frac {1}{m+1+g}\cdot\frac{n}{L}(1 \pm 2\eps_1)$$
            and $(U^\dagger_{k,1},\dots, U^\dagger_{k,m+1+g})$ is a $\left(2\eps_1,\left(d/4\right)^+,d/4\right)$-super-regular tuple;
        \item\label{(iii)dagger} for any $A,B\in V(R^\dagger)$, there are at least $dL$ indices $k\in[L_2]$ such that both $A$ and $B$ have at least $m-1+g$ $R^\dagger$-neighbours in $\{U_{k,1}^\dagger, U_{k,2}^\dagger,\dots, U_{k,m+1+g}^\dagger\}$.
    \end{enumerate}

    \smallskip

    \textbf{Step 3. Ensuring the divisibility of $|\Ub_k^\dagger|$.} 

    \smallskip

    In this step, we will reveal the edges of $G_2$. For any $k,q\in[L_2]$, $\ell\in[m+1+g]$, we can define a basic $(k,\ell;q)$-move similarly as in Step 2, that is, in such a move we remove a~copy of $K_r$ in $G \cup G_2$ that contains one vertex from $U_{k,\ell}^\dagger$ and $r-1$ vertices from $\Ub_{q}^\dagger$.
    Now fix a pair $(\Ub_k^\dagger, \Ub_{k'}^\dagger)$. We define a \emph{$(k,k')$-transfer} as follows. Choose an auxiliary index $q$ for which $U_{k,1}$ and $U_{k',1}$ have at least $m-1+g$ $R^\dagger$-neighbours in $\Ub_{q}^\dagger$. Perform one basic $(k,1;q)$-move and $r-1$ basic $(k',1;q)$-moves. Modulo $r$, this changes $|\Ub_k^\dagger|$ by -1, $|\Ub_{k'}^\dagger|$ by $+1$, and preserves $|\Ub_{q}^\dagger|$. Note that by \ref{(iii)dagger}, for any pair $k, k' \in [L_2]$, we can always find an index $q$ such that a $(k,k')$-transfer can be performed using $q$ as an auxiliary index. Using the same method as in the proof of Lemma~\ref{l:absorber-divisibility-by-r}, using no more than $rL_2$ $(k,k')$-transfers, we can pass to parts $U_{k,\ell}^\ddagger\subseteq U_{k,\ell}^\dagger$ with $\Ub_k^\ddagger := U_{k,1}^\ddagger \cup \dots \cup U_{k,m+1+g}^\ddagger$ satisfying $|\Ub_k^\ddagger| \equiv 0 \pmod{r}$. Note that the number of removed vertices in this step is at most $r^3L_2$ and none of them is a vertex from $\tilde X$.

    \smallskip
    
    \textbf{Step 4. Covering the remainder using $G_3$.}

    \smallskip

    We will reveal the graph $G_3$ separately for sets $S_{i,j}^\dagger$ and $U_{k,\ell}^\ddagger$.
    Consider first any $i \in [L_1]$. Recall that $(S_{i,1}^\dagger,S_{i,2}^\dagger,\dots, S_{i,m+1}^\dagger)$ is an $(\eps_1,(d/4)^+,d/4)$-super-regular $(m+1)$-tuple satisfying~\eqref{eq:correct_proprtion_in_S_dagger}. Thus, by Lemma~\ref{l:st-reg-pairs-factors}, w.h.p.~there is an $\tilde X$-conforming collection $\cK_i$ of vertex-disjoint $K_r$-cliques in $G \cup G_3$ covering  $(S_{i,1}^\dagger,S_{i,2}^\dagger,\dots, S_{i,m+1}^\dagger)$. Since the total number of such $(m+1)$-tuples is constant, w.h.p.~we can find such a collection $\cK_i$ for any $i\in[L_1]$. 
    
    Now consider $k\in[L_2]$. Since in Step 3 we removed at most $r^3L$ vertices from any set $U_{k,\ell}^\dagger$ to obtain sets $U_{k,\ell}^\ddagger$, the latter satisfy conditions analogous to \ref{(i)dagger}--\ref{(iii)dagger} with slightly perturbed parameters. In particular we can assume that $|U_{k,l}^\ddagger|=\frac{1}{m+1+g}\cdot\frac{n}{L}(1\pm3\eps_1)$ for each $k\in[L_2]$, $\ell\in[m+1+g]$ and that regularity, as well as super-regularity, is maintained.

    \smallskip

    \textit{The regular case ($g=0$).}

    \smallskip
       
    We first use Lemma~\ref{l:absorber-divisibility-by-r} and property (P1) to pass to an $(m+1)$-tuple $\{U_{k,1}^*, \ldots, U_{k,m+1}^*\}$ in which each part is divisible by $r$. To be more precise, Lemma~\ref{l:absorber-divisibility-by-r} gives a sequence of moves that needs to be performed to obtain the divisibility by $r$, while Remark~\ref{rem:(P1)} tells us that we can find a copy of $K_r$ in $G\cup G_3$ corresponding to a specific move. At this point we have removed at most a constant number of vertices from each part.
    
    Take $x$ to be the maximal integer such that all the parts have at least $xr$ vertices, in particular $x\geq \left\lfloor\frac1r\cdot\frac{1}{m+1}\cdot\frac{n}{L}(1-4\eps_1)\right\rfloor$. Let $|U_{k,\ell}^*| = xr+y_\ell$, noting that all the $y_\ell$ are divisible by $r$. We have 
    \[y_\ell\leq |U_{k,\ell}^*|-\frac{1}{m+1}\cdot\frac{n}{L}(1-4\eps_1)\leq 9\eps_1 xr.\] 
    Next, we partition each set $U_{k,\ell}^*$, $\ell\in[m+1]$, randomly into $m$ parts $U_{k,\ell,i}^*$, $i\in[m+1]\setminus\{\ell\}$, of size $xs$, and one part $U_{k,\ell,\ell}^*$ of size $xt+y_\ell$, so that by Lemma~\ref{l:slicing} all pairs $(U_{k,\ell,i}^*,U_{k,\ell',j}^*)$, $\ell\neq\ell'$, $i,j\in[m+1]$, are $(4r\eps_1, d')$-regular with $d' \geq d/(8r)$. Moreover, using an argument similar to that in the proof of Lemma~\ref{l:st-reg-pairs-factors}, w.h.p.~these pairs are $(4r\eps_1, d', d/(8r))$-super-regular with some $d' \geq d/(8r)$.
    
    Now consider an $(m+1)$-tuple $\{U^*_{k,1,\ell},U^*_{k,2,\ell},\ldots,U^*_{k,m+1,\ell}\}$. Using Lemma \ref{l:equalize-(m+1)-tuples-to-s/t} and Remark~\ref{rem:(P1)}, by removing an additional $4r\eps_1|\Ub_k^*|$ copies of $K_r$ in $G \cup G_3$, we can pass to an $(m+1)$-tuple $\{U_{k,1,\ell},U_{k,2,\ell},\ldots,U_{k,m+1,\ell}\}$ which satisfies the assumptions of Lemma~\ref{l:st-reg-pairs-factors}. Therefore, by Lemma~\ref{l:st-reg-pairs-factors}, w.h.p.~for each $k\in[L_2], \ell\in[m+1]$ there is a~collection of vertex-disjoint $K_r$-copies in $G \cup G_3$ covering the $(m+1)$-tuple $\{U_{k,1,\ell},U_{k,2,\ell},\ldots,U_{k,m+1,\ell}\}$.

    \smallskip

    \textit{The singular case $(g=1)$}

    \smallskip

    The proof is analogous to that in case $g=0$. First, using Lemma~\ref{l:sing-absorber-divisibility-by-s} and Remark~\ref{rem:(P1)} we can pass to an $(m+2)$-tuple $\{U_{k,1}^*, \ldots, U_{k,m+2}^*\}$ in which each part is divisible by~$s$. Next, using Lemma~\ref{l:sing-absorber-partitioning}, we can randomly partition each set $U_{k,\ell}^*$ into $(m+1)$-tuple $\{U_{k,\ell,1},\dots,U_{k,\ell,\ell-1},U_{k,\ell,\ell+1},\dots,U_{k,\ell,m+2}\}$, so that for each $\ell\in[m+2]$, all sets in the $(m+1)$-tuple $\{U_{k,1,\ell},\dots,U_{k,\ell-1,\ell},U_{k,\ell+1,\ell},\dots,U_{k,m+2,\ell}\}$ have the same size and all the involved pairs are $(4(m+1)\eps_1,d',d/(8r))$-super-regular with some $d'\geq d/(8r)$. By Lemma~\ref{l:st-reg-pairs-factors}, w.h.p.~there is a partial $K_r$-factor in $G\cup G_3$ covering these $(m+1)$-tuples. 
    \end{proof}

    We finish this section by providing a proof of Lemma~\ref{l:annoying-non-extremal}, which grants at least $\eta n$ vertex-disjoint copies of $K_{r+1}$ in $G \cup \Gnp$ under suitable assumptions. For simplicity, we deduce the lemma from Lemma~\ref{l:partition-reduced-graph}, but the only important ingredient is that in the singular case ($r = (m+1)s$), $G$ contains a regular $(m+2)$-tuple. There are several other ways to prove the lemma, and the constant $\eta$ is far from optimal.

    \begin{proof}[Proof of Lemma~\ref{l:annoying-non-extremal}]
    Let $G$ be the graph satisfying the assumption of the Lemma, with the given constants $r, s, \gamma$ and $\delta$. Suppose also that 
         $$\eta \ll \eps \ll d \ll \delta =\delta_2 \ll \gamma, r^{-1}.$$ 
        Apply Lemma~\ref{l:partition-reduced-graph} to $G$; specifically, $G$ contains an $(\epsilon, d)$-regular $(m+1+g)$-tuple $\{U_{k,1},\dots,U_{k,m+1+g}\}$. Recall that $g=1$ if $r = (m+1)s$, and $g=0$ otherwise. Apply Remark~\ref{rem:(P1)} (that is, property (P1)) $\eta n$ many times to find w.h.p.~the desired collection of vertex-disjoint $K_{r+1}$-copies in $G\cup \Gnp$.
    \end{proof}
    
    
    \section{The partitioning lemma}
	   \label{sec:partition-lemma} 
    
    In this section, we prove the partitioning lemma.
    
    \begin{proof}[Proof of Lemma~\ref{l:4.3chr}] 
    Let $\beta_0 = \frac{1}{10m(m+1)^2}$, and given $\beta\leq \beta_0$ take $\gamma_m \ll \beta$. 
    In particular, we need the following dependencies
    \begin{align*}
        \gamma_m \leq \frac{\beta^2}2, \quad \gamma_m \leq \sqrt{\frac{\beta^2(s/r-3\beta)}{2}}
    \end{align*}
    Let a graph $G$ and a sequence $\gamma_1, \dots, \gamma_m$ as in the statement of the Lemma be given.

    Let $h\in[m]$ be a maximal integer for which there exists a partition
    \begin{equation}\label{eq:partition} 
        V(G) = A_1^* \dcup \dots \dcup A_h^* \dcup A_{h+1}
    \end{equation}
    satisfying the following properties
    \begin{enumerate}[label=(\roman**)]
        \item\label{(i*)} For $i\in[h]$, $|A_i|=(s/r \pm \beta\gamma_i)n$ and $e(A_i^*)\leq \beta^2\gamma_i n^2$.
        \item\label{(ii*)}  Every vertex $x \in V(G) \setminus A_{h+1}$ has at most $4 \beta n$ non-neighbours in $A_{h+1}$.
        \item\label{(iii*)}  Every vertex $x \in A_{h+1}$ has at least $2\beta n$ neighbours in $A_i^*$, for any $i \in [h]$.
    \end{enumerate}
    Note that such $h$ exists, as $h=0$ trivially satisfies the above properties. Properties \ref{(i*)}--\ref{(iii*)} will help us to establish properties \ref{partition(i)}--\ref{partition(iii)} of the lemma. To get the remaining ones, we first deal with property \ref{partition(iv)}.

    \begin{claim}\label{claim:4.4chr}
        If $h<m$ and $X \subset A_{h+1}$ with $|X| =sn/r$, then $e(X) > \gamma_{h+1}^2 n^2$.
    \end{claim}
    \begin{proof}
        Assume for a contradiction that $h<m$ and there exists a subset $X \subseteq A_{h+1}$ with $|X| =sn/r$ and $e(X) \leq \gamma_{h+1}^2 n^2$. Note that 
        \begin{equation}\label{eq:size_of_X}
            |X| \geq n - \delta(G).
        \end{equation}
        We will show that we can change the initial partition (\ref{eq:partition}), keeping the conditions \ref{(i*)}--\ref{(iii*)}, but increasing $h$.

        Set $Y = A_{h+1}\setminus X$ and let
        \[ R = \{x\in X: |Y\setminus N(x)| \geq 3\beta n\}. \]
        Note that by (\ref{eq:size_of_X}) and $e(X) \leq \gamma_{h+1}^2n^2$, the number of missing edges between $X$ and $V(G)\setminus X$ is at most $2\gamma_{h+1}^2 n^2$.
        Moreover, since every vertex in $R$ contributes at least $3\beta n$ missing edges between $X$ and $V(G)\setminus X$, we have
        \[ |R| \leq \frac{2\gamma_{h+1}^2}{3\beta}n \leq \beta\gamma_{h+1}n.\]
        Similarly, we let
        \[ S = \{y\in Y: |N(y)\cap X| \leq 3\beta n\}, \]
        and note that
        \[ |S| \leq \frac{2\gamma_{h+1}^2}{s/r-3\beta}n\leq \frac{\beta^2\gamma_{h+1}}2 n.  \]
        Next, set
        \[ A_{h+1}^* = (X\setminus R) \cup S \quad \text{and} \quad A_{h+2} = (Y\setminus S)\cup R,\]
        and consider the partition
        \[ V(G) = A_1^* \dcup \dots \dcup A_{h+1}^* \dcup A_{h+2}. \]
        We claim that this new partition also satisfies \ref{(i*)}--\ref{(iii*)}, thus contradicting the maximality of $h$. 

        To show \ref{(i*)}, we only need to consider $i=h+1$. Since both sets $S$ and $R$ have size at most $\beta\gamma_{h+1}n$, we have $|A_{h+1}^*| \leq (s/r \pm \beta\gamma_{h+1})n$. Moreover,
        \[ e(A_{h+1}^*) \leq e(X) + |S|n \leq \left(\gamma_{h+1}^2 + \frac{\beta^2\gamma_{h+1}}2\right)n^2 \leq \beta^2\gamma_{h+1}n^2.\]

        To show \ref{(ii*)}, consider a vertex $x \in V(G)\setminus A_{h+2}$. If $x\notin A_{h+1}^*$, then since the initial partition (\ref{eq:partition}) satisfied \ref{(ii*)}, we know that $x$ has at most $4\beta n$ non-neighbours in $A_{h+1}$, thus also at most this many non-neighbours in $A_{h+2} \subset A_{h+1}$. Now suppose $x\in A_{h+1}^* = (X\setminus R)\cup S$. Due to $|A_{h+2}\setminus N(x)| \leq |Y \setminus N(x)| + |R|$ and $|R|\leq \beta n$, it suffices to show that $|Y\setminus N(x)| \leq 3\beta n$. If $x\in X\setminus R$ this is clear by the definition of $R$, so suppose that $x\in S$. Using (\ref{eq:size_of_X}), we get that $|X| \geq |V(G) \setminus N(x)| \geq |X \setminus N(x)| + |Y \setminus N(x)|$, hence, by the definition of $S$, $|Y\setminus N(x)| \leq |X\cap N(x)| \leq 3\beta n$.

        Finally, to show \ref{(iii*)}, let $x\in A_{h+2} = (Y\setminus S)\cup R$ and note that it is enough to consider $i=h+1$. Since $|A_{h+1}^* \cap N(x)| \geq |X \cap N(x)| - |R|$ and $|R| \leq \beta n$, it suffices to show that $|X \cap N(x)| \geq 3\beta n$. If $x\in Y\setminus S$ this is clear by the definition of $S$, so suppose that $x \in R$. Then, using again (\ref{eq:size_of_X}) and the definition of $R$, we have 
        \begin{align*} |X \cap N(x)| & = |X| + |N(x)| - |X \cup N(x)| \geq |V(G)| - |X \cup N(x)| \\
        & \geq |Y \setminus N(x)| \geq 3\beta n.
        \end{align*}
    \end{proof}

    We now adjust the partition to also get the remaining properties \ref{partition(v)} and \ref{partition(vi)}. To this end, choose $t\geq 0$ and a partition
    \[ V(G) \setminus A_{h+1} = A_1 \dcup \dots \dcup A_h\]
    such that
    \begin{enumerate}
        \item[(a)] $\sum_{i=1}^h |A_i^* \bigtriangleup A_i| \leq 2t$ 
        \item[(b)] $\sum_{i=1}^h e(A_i) \leq \sum_{i=1}^h e(A_i^*) - \beta tn$,
    \end{enumerate}
    and, subject to the above conditions, $t$ is maximal. Note that we can always find such a~partition and a parameter $t$, as $t=0$ and $A_i = A_i^*$ satisfies conditions (a) and (b). Moreover, conditions (b) and \ref{(i*)} imply that
    \[ 0 \leq m\beta^2\gamma_h n^2 - \beta t n,\]
    hence
    \begin{equation}\label{eq:t} 
        t \leq m\beta\gamma_h n \leq \frac{\beta}2 n.
    \end{equation}
    We claim that the partition
    \[ V(G) = A_1 \dcup \dots \dcup A_{h+1}\]
    has all desired properties \ref{partition(i)}--\ref{partition(vi)} of the lemma.

    To show \ref{partition(i)}, let $i\in[h]$. Then, using (a), \ref{(i*)} and (\ref{eq:t}), we get
    \begin{align}\label{eq:size_of_A_i}
        |A_i| & = |A_i^*| \pm 2t = \left(\frac{s}{r} \pm \beta\gamma_i\right)n \pm 2m\beta \gamma_hn \nonumber \\
        & = \left(\frac{s}{r} + 3\beta m \gamma_h\right) n = \left(\frac{s}{r} \pm \gamma_h\right)n.
    \end{align}

    Property \ref{partition(ii)} is equivalent to property \ref{(ii*)}.

    To show \ref{partition(iii)}, note that $|A_i^*\setminus A_i| \leq 2t \leq \beta n$, hence due to \ref{(iii*)}, every vertex $x\in A_{h+1}$ has at least $\beta n$ neighbours in any $A_i$.

    Property \ref{partition(iv)} is given by Claim~\ref{claim:4.4chr}.

    To show \ref{partition(v)}, take distinct indices $i,j\in[h]$. If there exists a vertex $x\in A_i$ with $|N(x) \cap A_j| < \frac13|A_j|$, then we can replace sets $A_i, A_j$ by sets $A_i'=A_i\setminus\{x\}, A_j'=A_j\cup\{x\}$, and reach a contradiction to the maximality of $t$. Indeed, for $k\in\{i, j\}$ we have $|A_k^* \bigtriangleup A_k'| \leq |A_k^* \bigtriangleup A_k| + 1$, and if $|N(x)\cap A_j| < \frac13|A_j|$, then $|N(x)\cap A_i| > \frac13|A_i|$, which in turn implies that $e(A_i')+e(A_j') < e(A_i)+e(A_j)$, so we can increase $t$.

    Finally, to show \ref{partition(vi)}, let $i,j\in[h+1]$ be distinct indices with $i < j$. Note that by \ref{(i*)}, (a) and (\ref{eq:t}) we have
    \begin{align*}
        e(A_i) \leq e(A_i^*) + |A_i\setminus A_i^*|n \leq (\beta^2\gamma_h+2\beta m\gamma_h)n^2 \leq 3\beta m \gamma_h n^2.
    \end{align*}
    Let $M$ be the number of missing edges between $A_i$ and $A_j$. Then
    \[ |A_i|\left(1-\frac{s}{r}\right)n \leq \sum_{x\in A_i} d(x) = 2e(A_i) + |A_i|(n-|A_i|) - M,\]
    hence, using (\ref{eq:size_of_A_i}), we can upper bound $M$ by
    \begin{align*}
        M & \leq |A_i| \left( \frac{s}{r}n - |A_i|\right) + 6\beta m\gamma_h n^2 \leq 3\beta m \gamma_h n (|A_i| + 2n) \leq \gamma_h|A_i||A_j|.
    \end{align*}

    \end{proof}
	

	\bibliographystyle{plain}

\end{document}